\documentclass[a4paper,10pt]{amsart}
\usepackage[a4paper,tmargin=40mm,bmargin=35mm,hmargin=30mm]{geometry}
\usepackage[utf8]{inputenc}
\usepackage[T1]{fontenc}

\usepackage{amsmath}
\usepackage{amssymb}
\usepackage{amsthm}
\usepackage{mathtools}
\usepackage{bm}
\usepackage{cancel}
\usepackage{eucal} 
\usepackage{mathdots}
\usepackage[colorlinks=true,linkcolor=red,citecolor=blue,urlcolor=green]{hyperref}
\usepackage[capitalise]{cleveref}
\usepackage{enumitem}
\usepackage{mathrsfs}  
\usepackage[normalem]{ulem}
\usepackage{tikz-cd}
\usepackage{adjustbox}
\usepackage{contour}
\usepackage{ulem}
\usepackage{faktor}

\numberwithin{equation}{section}
\crefname{equation}{}{}

\crefformat{section}{\S#2#1#3}
\crefmultiformat{section}{\S\S#2#1#3}{ and~#2#1#3}{, #2#1#3}{, and~#2#1#3}

\contourlength{0.8pt}

\newtheorem{theorem}{Theorem}[section]
\newtheorem{prop}[theorem]{Proposition}
\newtheorem{cor}[theorem]{Corollary}
\newtheorem{lemma}[theorem]{Lemma}
\newtheorem{question}[theorem]{Question}

\theoremstyle{definition}
\newtheorem{dfn}[theorem]{Definition}
\newtheorem{rmk}[theorem]{Remark}

\newtheorem{ex}[theorem]{Example}

\DeclareMathOperator{\Hom}{\mathsf{Hom}}

\DeclareMathOperator{\Ext}{\mathsf{Ext}}
\DeclareMathOperator{\Tor}{\mathsf{Tor}}

\newcommand{\op}{\mathsf{op}}

\newcommand{\lMod}[1]{#1\mbox{-}\mathsf{Mod}}

\newcommand{\Ker}{\mathsf{Ker}}
\newcommand{\Img}{\mathsf{Img}}
\newcommand{\Coker}{\mathsf{Coker}}

\newcommand{\Add}{\mathsf{Add}}

\newcommand{{\tst}}{\textit{t}-}
\newcommand{\pd}{\mathsf{pd}}
\newcommand{\fd}{\mathsf{fd}}
\newcommand{\id}{\mathsf{id}}

\newcommand{\wD}{\mathsf{w.D}}
\newcommand{\gD}{\mathsf{gl.D}}

\newcommand{\newterm}[1]{\textit{#1}}

\title[On $(n,d)$-Coherence and Its Applications to Algebraic $\mathsf{K}$-Theory]
{On $(n,d)$-Coherence and Its Applications to Algebraic $\mathsf{K}$-Theory}
\author{Rafael Parra}
\address[R. Parra]{IMERL, Facultad de Ingenier\'\i a, Universidad de la Rep\'ublica \\
Julio Herrera y Reissig 565, 11.300, Montevideo, Uruguay}
\email{rparra@fing.edu.uy}
\thanks{}

\begin{document}
\begin{abstract}
We introduce a unified approach to finiteness conditions in homological algebra and algebraic $\mathsf{K}$-theory by studying the class of $(n,d)$-coherent rings, defined for $n\in\mathbb{N}^*$ and $d\in\mathbb{N}^*\cup \{\infty\}$. This framework simultaneously controls higher finiteness conditions and bounds on the projective dimensions of modules. In the context of algebraic $\mathsf{K}$-theory, we prove that, for a left $(n,d)$-coherent ring, the inclusion $\mathsf{proj}(R)\hookrightarrow\mathsf{FP}_n^{\leq d}(R)$ induces isomorphisms on all nonnegative $\mathsf{K}$-groups. We then introduce the corresponding relative notions of $\mathsf{FP}_n^{\leq d}$-injective, $\mathsf{FP}_n^{\leq d}$-projective, $\mathsf{FP}_n^{\leq d}$-flat, and $\mathsf{FP}_n^{\leq d}$-cotorsion modules, and obtain characterizations of $(n,d)$-coherent rings in terms of these classes. When $d\geq\gD(R)$ or $d=\infty$, our notions recover several previously studied classes and results.
\end{abstract}

\maketitle
\section{Introduction}
Finitely generated and finitely presented modules over a ring $R$ play a fundamental role in homological algebra by providing the natural framework for defining finiteness conditions on rings, most notably the Noetherian and coherent conditions. These notions are closely related to algebraic $\mathsf{K}$-theory, where finiteness and homological dimensions play a central role. We refer to Section \ref{s:Finiteness conditions in modules of bounded projective dimension} for the notation used below. When $R$ is Noetherian (so that $\mathsf{FP}_0(R)=\mathsf{FP}_{\infty}(R)$) and regular (i.e., every finitely generated left ideal has finite projective dimension), the classical Resolution Theorem gives $\mathsf{K}_i(R)\cong\mathsf{K}_i(\mathsf{FP}_0(R))\cong\mathsf{K}_i(\mathsf{FP}_{\infty}(R))$ for all $i\ge 0$. In this setting, $R$ is $\mathsf{K_0}$-regular, i.e., $\mathsf{K}_0(R[t_1,\ldots,t_k])\cong\mathsf{K}_0(R)$ for all $k\geq 1$,
and its negative $\mathsf{K}$-groups vanish. This favorable behavior is closely related to the finiteness properties of Noetherian rings and their stability under polynomial extensions. The coherence setting provides a natural extension of this framework, where $\mathsf{FP}_1(R)=\mathsf{FP}_{\infty}(R)$. Unlike Noetherianity, however, coherence is not preserved by polynomial extensions in general. Nevertheless, coherent regular rings retain several of the features relevant to algebraic $\mathsf{K}$-theory. In this setting, $\mathsf{K}_i(R)\cong \mathsf{K}_i(\mathsf{FP}_1(R))\cong \mathsf{K}_i(\mathsf{FP}_{\infty}(R))$ for all $i\ge 0$, $\mathsf{K}_{-1}(R)=0$ by \cite[Thm. 3.30]{AGH19}, and $\mathsf{K}_i(R)\cong\mathsf{K}_i(R[t_1,\cdots ,t_k])$ for $i\ge k-1$ by \cite[Cor. 5.3]{BL23}. The regularity condition, however, can be difficult to verify for coherent rings, particularly when the weak dimension is infinite. One of the results of this paper gives a characterization in terms of direct limits:
\begin{theorem}[Corollary \ref{cor: coh regular} below]
Let $R$ be a ring. Then every cyclic $R$-module is a direct limit of modules in $\mathsf{CFP}_\infty^{<\infty}(R)$ if and only if $R$ is left coherent regular.
\end{theorem}

A natural question is whether these results admit a common generalization. This was one of the motivations of \cite{EP22}, where the authors considered $n$-coherent rings, namely rings for which every finitely $n$-presented module is finitely $(n+1)$-presented, or equivalently, $\mathsf{FP}_n(R)=\mathsf{FP}_{\infty}(R)$. Adding the corresponding regularity condition, namely $\mathsf{FP}_n(R) = \mathsf{FP}_n^{<\infty}(R)$, gives rise to the notions of $n$-coherent and $n$-regular rings. For such rings, one has $\mathsf{K}_i(R)\cong\mathsf{K}_i(\mathsf{FP}_n(R))\cong\mathsf{K}_i(\mathsf{FP}_{\infty}(R))$ for all $i\ge 0$. However, the behavior of negative $\mathsf{K}$-groups is different: \cite{PA25} provides an example of a $2$-coherent and $2$-regular ring that is $\mathsf{K}_0$-regular yet satisfies $\mathsf{K}_{-1}(R)\neq 0$.

Up to this point, the standard approach has required combining two separated requirements: a finiteness condition on the ring and a bound on the projective dimension of certain classes of modules. The main objective of this paper is to study a unified class of rings that incorporates both aspects into a single framework depending on two integer parameters. Our strategy is further motived by the notion of $(n,d)$-coherent rings introduced by Mao and Ding \cite{MD07}: for $n\in\mathbb{N}^*$ and $d\in\mathbb{N}^*\cup{\infty}$, a ring $R$ is called \newterm{$(n,d)$-coherent} if every finitely $n$-presented $R$-module of projective dimension at most $d$ is finitely $(n+1)$-presented. This framework simultaneously controls the relevant finiteness and projective dimension conditions and provides a natural setting in which to study their interaction with algebraic $\mathsf{K}$-theory.

\begin{theorem}[Proposition \ref{prop: k} below]
Let $R$ be a left $(n,d)$-coherent ring, with $n,d\in\mathbb{N}^*$. Then $\mathsf{K}_i(R)\cong \mathsf{K}_i(\mathsf{FP}_n^{\leq d}(R))$ for all $i\geq0$.
\end{theorem}

Another contribution of this work is a unified approach to finiteness conditions in homological algebra via the classes  $\mathsf{FP}_n^{\leq d}(R)$. These classes naturally induce the notions of $\mathsf{FP}_n^{\le d}$-injective and $\mathsf{FP}_n^{\le d}$-flat modules, defined by the vanishing of $\Ext_R^1(M,-)$ and $\Tor_1^R(-,M)$, respectively, for all $M \in \mathsf{FP}_n^{\le d}(R)$. Using this framework, we characterize $(n,d)$-coherent rings in terms of the relative injective, projective, flat, and cotorsion modules. Moreover, if $d\ge\gD(R)$ or $d=\infty$, then $(n,d)$-coherence coincides with $n$-coherence, and the corresponding relative notions recover those previously studied in the literature; see \cite{Zhou04, Zhu, BP, Zhu17, gp2}. 

The paper is organized as follows. In Section \ref{s:Finiteness conditions in modules of bounded projective dimension}, we investigate the closure properties of the class of finitely $n$-presented 
modules of bounded projective dimension. We also introduce the notion of $(n,d)$-coherent modules. When $d\ge\gD(R)$ or $d=\infty$, this notion coincides with that of $n$-coherent modules 
studied in \cite{DKM97}. Section \ref{s:$(n,d)$-coherent rings} is devoted to the study of 
$(n,d)$-coherent rings. We establish several characterizations in terms 
of finitely $n$-presented modules of projective dimension at most $d$ 
as well as in terms of $(n,d)$-coherent modules. We further examine the behavior 
of the $(n,d)$-coherence property under various constructions, including 
surjective ring homomorphisms and skew group rings. In Section \ref{s: k}, we investigate the connections between $(n,d)$-coherence and algebraic $\mathsf{K}$-theory, with particular emphasis on coherent regular rings.  Finally, in Sections \ref{s:inj and proj relativos} and \ref{s:planos y cotorsion relativos}, we introduce and study the classes of $\mathsf{FP}_n^{\le d}$-injective, $\mathsf{FP}_n^{\le d}$-projective, 
$\mathsf{FP}_n^{\le d}$-flat, and $\mathsf{FP}_n^{\le d}$-cotorsion modules over an arbitrary ring. We then employ these relative homological notions to obtain further characterizations of $(n,d)$-coherent rings.

\textbf{Conventions and Notation:} Throughout this article, $R$ denotes an associative ring with identity, $\lMod R$ the category of left $R$-modules, and $R^{\op}$ the opposite ring of $R$. Right $R$-modules are identified with left $R^{\op}$-modules. Unless otherwise specified, the terms “module” and “ideal” mean left $R$-module and left ideal, respectively. We write $\mathbb{N}^* := \mathbb{N} \setminus \{0\}$ for the set of positive integers. Monomorphisms and epimorphisms are denoted by 
$\rightarrowtail$ and $\twoheadrightarrow$, respectively. For $\mathcal{X} \subseteq \lMod R$, we define
$$\mathcal{X}^{\perp_1} \coloneqq \{M\in\lMod R \mid\Ext^1_R(X, M)=0 \text{ for all } X \in \mathcal{X} \},$$
$$\mathcal{X}^{\perp_{\geq i}} \coloneqq \{ M\in \lMod R \mid\Ext^j_R(X, M)=0 \text{ for all } X\in\mathcal{X},\, j \geq i \},$$
and set $\mathcal{X}^{\perp} \coloneqq \mathcal{X}^{\perp_{\geq 1}}$. The classes $ {}^{\perp_1}\mathcal{Y}$, ${}^{\perp_{\geq i}}\mathcal{Y}$, and ${}^{\perp}\mathcal{Y}$ for $ \mathcal{Y}\subseteq \lMod R$ are defined in the same way.  Let $\mathcal{X} \subseteq \lMod{R^{\op}}$. Define
$$\mathcal{X}^{\top_1}:= \{M\in\lMod R\mid\Tor^R_1(X,M)=0 \text{ for all } X \in \mathcal{X}\},$$
and, for $i \ge 1$, $ \mathcal{X}^{\top_{\ge i}}:=\{ M\in\lMod R\mid\Tor^R_j(X,M)=0 \text{ for all } X\in\mathcal{X}\text{ and all } j \ge i\}.$ For brevity, we set $\mathcal{X}^{\top}:=\mathcal{X}^{\top_{\ge 1}}$. 

\section{Finitely $n$-presented modules of bounded projective dimension and $(n,d)$-coherent modules}\label{s:Finiteness conditions in modules of bounded projective dimension} 

In what follows, for each $d\in\mathbb{N}$, we denote by $\mathsf{P}^{\le d}(R)$ the class of all $R$-modules of projective dimension at most $d$. This class is closed under direct summands and is \newterm{resolving}; that is, it contains all projective modules and is closed under extensions and kernels of epimorphisms in $\lMod R$. We write $\mathsf{P}^{<\infty}(R):=\bigcup_{d\in\mathbb{N}}\mathsf{P}^{\le d}(R)$ for the class of $R$-modules of finite projective dimension. This class is \newterm{thick}; that is, it is closed under direct summands and satisfies the two-out-of-three property for short exact sequences.

Let $n\in\mathbb{N}$. An $R$-module $M$ is called \newterm{finitely $n$-presented} if there exists an exact sequence $F_n\to F_{n-1}\to\cdots\to F_1\to F_0\twoheadrightarrow M$, where each $F_i$ is a finitely generated projective (equivalently, free) $R$-module for $0\le i\le n$. Such a sequence is is called a \newterm{finite $n$-presentation} of $M$. The class of all finitely $n$-presented $R$-modules is denoted by $\mathsf{FP}_n(R)$. In particular, $\mathsf{FP}_0(R)$ is the class of all finitely generated $R$-modules while $\mathsf{FP}_1(R)$ is the class of all finitely presented $R$-modules. Moreover, the class $\mathsf{FP}_\infty(R):=\bigcap_{n \ge 0} \mathsf{FP}_n(R)$ consists of all $R$-modules admitting a projective resolution by finitely generated modules. Modules in $\mathsf{FP}_\infty(R)$ are called \newterm{finitely $\infty$-presented} $R$-modules. Clearly, $\mathsf{proj}(R)\subseteq \mathsf{FP}_\infty(R),$ where $\mathsf{proj}(R)$ denotes the class of finitely generated projective $R$-modules. For convenience, we also set $\mathsf{FP}_{-1}(R):=\lMod R$. These classes form the descending chain $\mathsf{proj}(R)\subseteq \mathsf{FP}_\infty(R) \subseteq\cdots\subseteq\mathsf{FP}_n(R)\subseteq\cdots\subseteq \mathsf{FP}_1(R)\subseteq\mathsf{FP}_0(R)$. The classes $\mathsf{FP}_n(R)$ satisfy several closure properties. In particular, $\mathsf{FP}_n(R)$ is closed under extensions, direct summands, and cokernels of monomorphisms between its objects. In contrast, $\mathsf{FP}_\infty(R)$ is a thick class; see \cite[Prop. 1.7, Thm. 1.8]{BP}. We now consider finitely $n$-presented modules of bounded projective dimension. For $n,d\in\mathbb{N}$, set
$$\mathsf{FP}_n^{\le d}(R):=\mathsf{FP}_n(R)\cap\mathsf{P}^{\le d}(R),
\qquad\mathsf{FP}_\infty^{\le d}(R):=\mathsf{FP}_\infty(R)\cap\mathsf{P}^{\le d}(R).$$
For modules of finite projective dimension, set
$$\mathsf{FP}_n^{<\infty}(R):=\mathsf{FP}_n(R)\cap\mathsf{P}^{<\infty}(R),
\qquad \mathsf{FP}_\infty^{<\infty}(R):=\mathsf{FP}_\infty(R)\cap \mathsf{P}^{<\infty}(R).$$
When $d=\infty$, we adopt the convention
$$\mathsf{FP}_n^{\leq\infty}(R):=\mathsf{FP}_n(R),
\qquad \mathsf{FP}_\infty^{\leq\infty}(R):=\mathsf{FP}_\infty(R).$$

For the purposes of this paper, the following table summarizes the closure properties of the
classes under consideration.

\begin{table}[h]
\begin{tabular}{lcccc}
\hline
Class & Extensions & Direct summands & Kernels of epis. & Cokernels of monos.\\
\hline
$\mathsf{FP}_n^{\leq d}(R)$ & $\checkmark$ & $\checkmark$ &  &\\
$\mathsf{FP}_n^{\leq d}(R)$, $d\leq n$ & $\checkmark$ & $\checkmark$ & $\checkmark$ &\\
$\mathsf{FP}_n^{<\infty}(R)$ & $\checkmark$ & $\checkmark$ & & $\checkmark$\\
$\mathsf{FP}_\infty^{\leq d}(R)$ & $\checkmark$ & $\checkmark$ & $\checkmark$ &\\
$\mathsf{FP}_\infty^{<\infty}(R)$ & $\checkmark$ & $\checkmark$ & $\checkmark$ & $\checkmark$\\
\hline
\end{tabular}
\end{table}

Following Dobbs, Kabbaj, and Mahdou \cite{DKM97}, for each $n\in\mathbb{N}^*$, an $R$-module $M$ is called \newterm{$n$-coherent} if $M\in\mathsf{FP}_n(R)$ and every finitely $(n-1)$-presented submodule of $M$ is finitely $n$-presented. Motivated by this notion, we introduce the following definition.

\begin{dfn}\label{dfn: mod coh}
Let $n\in\mathbb{N}^*$ and $d\in\mathbb{N}^*\cup\{\infty\}$. An $R$-module
$M$ is called \newterm{$(n,d)$-coherent} if $M\in\mathsf{FP}_n(R)$ and every
finitely $(n-1)$-presented submodule $N$ of $M$ with
$\pd_R(N)\le d-1$ is finitely $n$-presented.
\end{dfn}

We denote by $\mathsf{(n,d)\text{-}Coh}(R)$ the class of all $(n,d)$-coherent
$R$-modules. When $d=\infty$, this notion coincides with $n$-coherence
in the sense of \cite{DKM97}. Moreover, the following closure property holds:
if $M$ is $(n,d)$-coherent and $N\subseteq M$ is finitely $(n-1)$-presented with $\pd_R(N)\le d-1$, then
$N$ itself is $(n,d)$-coherent.

\begin{theorem}\label{theo: (n,d)-mod}
Let $n\in\mathbb{N}^*$ and $d\in\mathbb{N}^*\cup\{\infty\}$.  
Consider a short exact sequence of $R$-modules $
P \overset{f}{\rightarrowtail} N\overset{g}{\twoheadrightarrow} M$,
where $N$ is $(n,d)$-coherent. Then the following hold:
\begin{enumerate}\renewcommand{\labelenumi}{(\alph{enumi})}
\item If $P\in\mathsf{FP}_{n-1}^{\le d-1}(R)$, then both $M$ and $P$ are $(n,d)$-coherent.
\item If $M\in\mathsf{FP}_n(R)$ and $\pd_R(P)\le d-1$, then $P$ is $(n,d)$-coherent.
\end{enumerate}
\end{theorem}
\begin{proof}~\
\begin{itemize}
\item For part \textbf{(a)}, we first prove that $M$ is $(n,d)$-coherent. Clearly, $M\in\mathsf{FP}_n(R)$. Let $M_1$ be a finitely $(n-1)$-presented submodule of $M$ with $\pd_R(M_1) \le d-1$. Consider the
short exact sequence $P\overset{f}{\rightarrowtail} g^{-1}(M_1)\overset{g}{\twoheadrightarrow} M_1$. Since $g^{-1}(M_1)\subseteq N$ and $N$ is $(n,d)$-coherent, we conclude that $g^{-1}(M_1)$ is finitely $n$-presented. Consequently, $M_1$ is finitely $n$-presented. The fact that $P$ is $(n,d)$-coherent follows from the preceding comments.
\item Part \textbf{(b)} is immediate. 
\end{itemize}
\end{proof}

When $d=\infty$, we recover \cite[Thm. 2.3]{DKM97}. Moreover, the same argument as in the proof of \cite[Thm. 2.4]{DKM97} yields the following corollary.

\begin{cor}
Let $n\in\mathbb{N}^*$ and $d \in \mathbb{N}^* \cup \{\infty\}$.  
Let $M_0\xrightarrow{f_1} M_1 \xrightarrow{f_2}\cdots\xrightarrow{f_n} M_n$
be an exact sequence of $(n,d)$-coherent $R$-modules such that
$\pd_R(\Img(f_i))\le d-1$ for all $i=1,\ldots,n$. Then $\Img(f_i)$ and $\Coker(f_i)$ are $(n,d)$-coherent for every $i=1,\ldots,n$.
\end{cor}

Applying Definition \ref{dfn: mod coh} to the $R$-module $R$ leads to the following observation.

\begin{rmk}
Let $R$ be a ring, $n\in\mathbb{N}^*$, and $d\in\mathbb{N}^*\cup{\infty}$. We denote by $\mathsf{CFP}_n^{\leq d}(R)$ the class of all cyclic $R$-modules in $\mathsf{FP}_n^{\le d}(R)$.
\begin{enumerate}\renewcommand{\labelenumi}{(\alph{enumi})}
\item We say that $R$ is left \newterm{weakly $(n,d)$-coherent} if every  finitely $(n-1)$-presented ideal $I$ of $R$ satisfying $\pd_R(I)\le d-1$ is finitely $n$-presented. Equivalently, $\mathsf{CFP}_n^{\le d}(R)\subseteq \mathsf{FP}_{n+1}(R)$. When $d=\infty$ (with the convention $\mathsf{CFP}_n^{\le\infty}(R)=\mathsf{CFP}_n(R)$), this recovers the usual notion of left weakly $n$-coherent ring. For $n\ge 2$, it is not known whether this condition is equivalent to left $n$-coherence. Recall that a ring $R$ is left \newterm{$n$-coherent} if every finitely $n$-presented $R$-module is finitely $(n+1)$-presented.
\item If $R$ is left weakly $(n,d)$-coherent, then $\mathsf{CFP}_n^{\le d}(R)\subseteq (n,d)\text{-}\mathsf{Coh}(R)$. When $d=\infty$, this recovers \cite[App. 2.10]{DKM97}.
\item By \cite[Thm. 2.3]{B10}, applied to the class $\mathcal{X}=\mathsf{CFP}_n^{\le d}(R)$, the condition in (a) is equivalent to each of the following:
\begin{enumerate}
\item[(i)] For every set $J$ and every $M\in\mathsf{CFP}_n^{\le d}(R)$, the canonical morphism
$R^J\otimes_R M\to M^J$ is bijective, and $\Tor_i^R(R^J,M)=0$ for all $1\leq i\leq n$.
\item[(ii)] For every family $(P_j)_{j\in J}$ of $R^{\op}$-modules and every $M\in\mathsf{CFP}_n^{\le d}(R)$, the canonical morphism
$\prod_{j\in J}\Tor_i^R(P_j,M)\to
\Tor_i^R\left(\prod_{j\in J}P_j,M\right)$
is bijective for all $i\leq n$.
\item[(iii)] For every $M\in\mathsf{CFP}_n^{\le d}(R)$ and every directed system $(N_j)_{j\in J}$ of $R$-modules, the canonical morphism $\varinjlim\Ext_R^i(M,N_j)\to \Ext_R^i\left(M,\varinjlim N_j\right)$
is bijective for all $i\leq n$.
\end{enumerate}
\item A more detailed treatment of the case $n=1$ is given in \cite{PP26}. 
\end{enumerate}
\end{rmk}

\section{$(n,d)$-coherent rings}\label{s:$(n,d)$-coherent rings}

According to Mao and Ding \cite[Def. 4.1]{MD07}, a ring $R$ is left \newterm{$(n,d)$-coherent}, where $n\in\mathbb{N}^*$ and $d\in\mathbb{N}^*\cup{\infty}$, if every finitely $n$-presented $R$-module of projective dimension at most $d$ is finitely $(n+1)$-presented. Following Lee \cite{Lee02}, a ring $R$ is said to be left \newterm{Lee $d$-coherent} if every finitely generated submodule of a free $R$-module with projective dimension at most $d-1$ is finitely presented. A natural extension of this condition to $n\ge 2$ is given by requiring that every finitely $(n-1)$-presented submodule of a free $R$-module with projective dimension at most $d-1$ be finitely $n$-presented. The next proposition shows that this condition is equivalent to left $(n,d)$-coherence.

\begin{prop}\label{prop: car1}
Let $R$ be a ring, $n\in\mathbb{N}^*$, and $d\in \mathbb{N}^*\cup\{\infty\}$.
Then the following statements are equivalent:
\begin{enumerate}
    \item[(1)] $R$ is left $(n,d)$-coherent.
    \item[(2)] Every finitely $(n-1)$-presented submodule of a free $R$-module with projective dimension at most $d-1$ is finitely $n$-presented.
\end{enumerate}
\end{prop}
\begin{proof} ~\
\begin{itemize}
\item (1) $\Rightarrow$ (2): Let $K$ be a finitely $(n-1)$-presented submodule of a  free $R$-module $F$ (without loss of generality, finitely generated) with $\pd_R(K) \le d-1$. Then the quotient $F/K$ is finitely $n$-presented and satisfies $\pd_R(F/K) \le d$. By assumption, $F/K$ is finitely $(n+1)$-presented, which implies that $K$ is finitely $n$-presented.
\item (2) $\Rightarrow$ (1): Let $M$ be a finitely $n$-presented $R$-module with $\pd_R(M) \le d$. Then there exists a short exact sequence $K \rightarrowtail F \twoheadrightarrow M$,
where $F$ is a finitely generated free $R$-module and $K$ is finitely
$(n-1)$-presented. Since $\pd_R(M)\le d$, we have $\pd_R(K) \le d-1$. By $(2)$, $K$ is finitely $n$-presented, and hence $M$ is finitely $(n+1)$-presented. 
\end{itemize}
\end{proof}

In particular, $R$ is left Lee $d$-coherent if and only if it is left $(1,d)$-coherent. One of the aims of this paper is to extend the results of \cite{Lee02} to the case $n \ge 2$. Notre that if $d\le n$, then $\mathsf{FP}_n^{\le d}(R)\subseteq \mathsf{FP}_{\infty}(R)$. Consequently, every ring $R$ is left $(n,d)$-coherent in this case. Thus, throughout the paper, we assume $d>n$. 

\begin{prop}
Let $R$ be a ring, $n\in \mathbb{N}^*$, and $d\in \mathbb{N}^* \cup \{\infty\}$.
Then the following statements are equivalent:
\begin{enumerate}
\item[(1)] $R$ is left $(n,d)$-coherent.
\item[(3)] Every finitely $n$-presented $R$-module of projective dimension at most $d$ is $(n,d)$-coherent.
\item[(4)] Every finitely generated projective $R$-module is $(n,d)$-coherent.
\item[(5)] Every finitely generated free $R$-module is $(n,d)$-coherent.
\end{enumerate}
\end{prop}
\begin{proof}
The implication $(1)\Rightarrow(3)$ follows directly from the definition. The implications $(3)\Rightarrow (4)\Rightarrow(5)$ are immediate, and $(5)\Rightarrow(1)$ follows from Proposition \ref{prop: car1}.
\end{proof}

In particular, when $d\ge \gD(R)$ or $d=\infty$, we recover \cite[Prop. 2.5]{EP22}.

\begin{cor}\label{cor: coh}
Let $R$ be a ring, $n\in\mathbb{N}^*$, and $d\in\mathbb{N}^*\cup\{\infty\}$.
Then the following statements are equivalent:
\begin{enumerate}
\item[(1)] $R$ is left $(n,d)$-coherent.
\item[(6)] $\mathsf{FP}_n^{\le d}(R)= \mathsf{(n,d)}\text{-}\mathsf{Coh}(R)\cap\mathsf{P}^{\le d}(R)$.
\end{enumerate}
\end{cor}

\begin{rmk}\label{rmk: 1}
Let $n \in \mathbb{N}^*$ and $d\in \mathbb{N}^*\cup \{\infty\}$.
\begin{enumerate}\renewcommand{\labelenumi}{(\alph{enumi})}
\item Let $R$  be a ring. If $R$ is left $(n,d)$-coherent, then it is left $(m,d)$-coherent for every $m>n$, and left $(n,s)$-coherent for every $s<d$. Moreover, if $R$ is left $n$-coherent, then it is left $(n,d)$-coherent for every
$d\in\mathbb{N}^*$. Conversely, if $R$ is left $(n,d)$-coherent for some $d\ge\gD(R)$, then $R$ is left $n$-coherent. 
Finally, if $\wD(R)=k<\infty$, then $R$ is left $(k+1,d)$-coherent for every $d\in \mathbb{N}^*$ by \cite[Cor. 5.8]{L25}.
\item Every left $(n,d)$-coherent ring is left weakly $(n,d)$-coherent. 
\item Rings for which every finitely $n$-presented module has projective dimension bounded by some $t\in\mathbb{N}$ are called left \newterm{$(n,t)$-rings}; see \cite{Costa}. Such rings are left $\max\{n,t\}$-coherent. Consequently, they are left $(n,d)$-coherent for all $d$ if $n\ge t$, and left $(t,d)$-coherent for every $d$ if $n<t$.
\item There exist rings that are not left $(n,d)$-coherent for any $d>n$. Let $n\ge 2$, and let $R$ be the ring constructed in \cite[Ex. 5.12]{L25}. Although modules are treated on the right in that work, the same construction applies to left modules after the necessary adjustments.  This ring satisfies $\mathsf{FP}_n(R) \subseteq \mathsf{P}^{\le n+1}(R)$, 
and is not left $n$-coherent. Thus, $R$ is not left $(n,d)$-coherent for any $d>n$. On the other hand, $R$ is a left $(n,n+1)$-ring
and hence is left $(n+1,d)$-coherent for every $d\in\mathbb{N}^*$.
\item Left $(n,d)$-coherence is not, in general, preserved under quotients. Indeed, let $R$ be a commutative ring with $\gD(R)=3$ that is not coherent. Then $R$ is not $(1,3)$-coherent. On the other hand, let
$X$ be a (possibly infinite) set of commuting indeterminates. Since
$\mathbb{Z}[X]$ is coherent, it is $(1,d)$-coherent for every
$d\in\mathbb{N}^*$. Moreover, every commutative ring is isomorphic to a
quotient of $\mathbb{Z}[X]$ for a suitable set $X$; see
\cite[Chap. 4, \S4.G, Ex. 4.61, p. 143]{L99}. 
\end{enumerate}
\end{rmk}

\begin{ex}\label{ex: 1} 
There exist non-$1$-coherent rings that are $(1,d)$-coherent for every $d\in \mathbb{N}^*$.
Let $\mathsf{k}$ be a field, and consider the local ring $R=\mathsf{k}[x_1, x_2, \ldots]/(x_1, x_2, \ldots)^2$
with maximal ideal $\mathfrak{m}=(x_1, x_2, \ldots)$. This ring is not coherent, since the kernel of the canonical epimorphism $R\twoheadrightarrow (x_1)$ is not finitely generated; see \cite[Ex. 1.3]{BP}. Moreover, every non-projective $R$-module has infinite projective dimension; see \cite[Ex. 2.2]{CET}. Hence,
$\mathsf{FP}_1^{\le d}(R)=\mathsf{proj}(R)$ for every $d\in\mathbb{N}^*$, and consequently $R$ is left $(1,d)$-coherent for all $d\in\mathbb{N}^*$.
\end{ex}

\begin{prop}\label{prop:equiv}
Let $R$ be a ring, $n\in\mathbb{N}^*$, and $d\in \mathbb{N}^*\cup \{\infty\}$. The following statements are equivalent:
\begin{enumerate}
\item[(1)] $R$ is left $(n,d)$-coherent.
\item[(7)] $\mathsf{FP}_n^{\le d}(R)=\mathsf{FP}_{n+j}^{\le d}(R)$  for every $j \in \mathbb{N}^*$.
\item[(8)] $\mathsf{FP}_n^{\le d}(R)=\mathsf{FP}_{\infty}^{\le d}(R)$.
\item[(9)] The class $\mathsf{FP}_n^{\le d}(R)$ is closed under kernels of epimorphisms between its objects.
\end{enumerate}
\end{prop}
\begin{proof} By part (a) of Remark \ref{rmk: 1}, we have
$(1)\Rightarrow(7)$. Since
$\mathsf{FP}_{\infty}^{\le d}(R)
=\bigcap_{i\ge 1}\mathsf{FP}_i^{\le d}(R),$
we have $(7)\Rightarrow(8)$. The implications $(8)\Rightarrow(9)$ and $(9)\Rightarrow(1)$ are immediate.
\end{proof}

When $d\ge \gD(R)$, we recover \cite[Thm. 2.4 and Cor. 2.6]{BP}. The following characterization is an
immediate consequence of \cite[Cor. 2]{St76} and Proposition \ref{prop:equiv}.

\begin{cor}
Let $R$ be a ring, and let $n,d\in\mathbb{N}^*$. The following statements are equivalent:
\begin{enumerate}
\item[(1)] $R$ is left $(n,d)$-coherent.
\item[(10)] For every $M\in\mathsf{FP}_n^{\leq d}(R)$ and every index set $I$, the canonical map $\bigoplus_{i\in I}\Ext_R^k(M,R)\to\Ext_R^k\left(M,\bigoplus_{i\in I}R\right)$ is surjective for every $0\leq k\leq d$.
\end{enumerate}
\end{cor}

The following characterization of left $(n,d)$-coherent rings follows from \cite[Thm. 2.3]{B10} by taking $\mathcal{X}=\mathsf{FP}_n^{\le d}(R)$.

\begin{prop}
Let $R$ be a ring, $n\in\mathbb{N}^*$, and $d\in\mathbb{N}^*\cup\{\infty\}$.
The following statements are equivalent:
\begin{enumerate}
\item[(1)] $R$ is left $(n,d)$-coherent.
\item[(10)] For every set $J$ and every $M\in\mathsf{FP}_n^{\le d}(R)$, the canonical morphism
$R^{J}\otimes_R M\to M^{J}$
is bijective and $\Tor^R_i(R^{J},M)=0$ for all $1\le i\le n$.
\item[(11)] For every family $\{P_j\}_{j\in J}$ of $R^{\op}$-modules and every
$M\in\mathsf{FP}_n^{\le d}(R)$, the canonical morphism
$\prod_{j\in J}\Tor^R_i(P_j,M)\to
\Tor^R_i\!\left(\prod_{j\in J}P_j,M\right)$
is bijective for all $ i\le n$.
\item[(12)] For every $M\in\mathsf{FP}_n^{\le d}(R)$ and every directed system
$\{N_j\}_{j\in J}$ of $R$-modules over a directed index set $J$, the canonical morphism
$\varinjlim \Ext^i_R(M,N_j)\to
\Ext^i_R\!\left(M,\varinjlim N_j\right)$
is bijective for all $ i\le n$.
\end{enumerate}
\end{prop}

The following corollary is an immediate consequence of the preceding proposition and Remark \ref{rmk: 1}. When $d\ge \gD(R)$, it recovers the equivalence between conditions (1) and (7) in \cite[Thm. 4.1]{MD06}.

\begin{cor} Let $R$ be a ring, $n\in\mathbb{N}^*$ and $d\in\mathbb{N}^*\cup\{\infty\}$.  
The following statements are equivalent:
\begin{enumerate}
\item[(1)] $R$ is left $(n,d)$-coherent.
\item[(13)] For every directed system $\{N_j\}_{j \in J}$ of $R$-modules over a directed index set $J$ and every $M \in \mathsf{FP}_n^{\le d}(R)$, the canonical homomorphism $\varinjlim \Ext^i_R(M, N_j)\to\Ext^i_R\!\left(M,\varinjlim N_j\right)$ is bijective for all $i\geq 0$.
\end{enumerate}    
\end{cor}

Let $R_1$ and $R_2$ be rings, and set $R = R_1 \oplus R_2$. Let $M = M_1 \oplus M_2$ be an $R$-module, where each $M_i$ is an $R_i$-module for $i=1,2$. By parts (1) and (2) of \cite[Lem. 6.4]{L25}, we have
$\pd_R(M)=\sup\{\pd_{R_1}(M_1),\,\pd_{R_2}(M_2)\},$ and $M \in \mathsf{FP}_n(R)$ if and only if $M_i \in \mathsf{FP}_n(R_i)$ for $i=1,2$. It follows that the class of left $(n,d)$-coherent rings is closed under finite direct products.

\begin{theorem}\label{theo: products}
Let $n\in\mathbb{N}^*$ and $d\in\mathbb{N}^*\cup \{\infty\}$, and let $R_1,\ldots,R_k$ be rings. Set $R=\prod_{i=1}^k R_i$. Then $R$ is left $(n,d)$-coherent if and only if each $R_i$ is left
$(n,d)$-coherent for $i=1,\ldots,k$.
\end{theorem}

Theorem \ref{theo: products} gives more examples of $(n,d)$-coherent rings. Let $\mathsf{k}$ be a field, and let
$R_1=\mathsf{k}[x_1,x_2,\ldots]/(x_1,x_2,\ldots)^2$ be the ring introduced in Example \ref{ex: 1}. Let $R_2$ be any coherent ring. By Theorem \ref{theo: products}, the product ring $R=R_1\times R_2$ is $(1,d)$-coherent for every $d\in\mathbb{N}^*$. On the other hand, $R$ is not coherent by \cite[Thm. 2.13]{DKM97}.

\begin{rmk}\label{rmk:surjections-skew}
Let $n\in\mathbb{N}^*$ and $d\in\mathbb{N}^*\cup\{\infty\}$.
\begin{enumerate}\renewcommand{\labelenumi}{(\alph{enumi})}
\item By \cite[Lem. 3.2]{ODL14}, for any surjective homomorphism of rings $R\twoheadrightarrow S$ such
that $S$ is projective as both an $R$-module and an $R^{\op}$-module, if
$R$ is left $(n,d)$-coherent, then so is $S$. 
\item Let $G$ be a finite group acting on a ring $R$ via $\sigma \colon G \to \mathrm{Aut}(R)$, let $H \le G$, and let $R_\sigma[G]$ and $R_\sigma[H]$ be the associated skew group rings; see \cite{GP25}. If $R_\sigma [G]$ is left $(n,d)$-coherent, then so is $R_\sigma [H]$. Moreover, if $R_\sigma [G]$ is a separable extension of $R_\sigma [H]$, then $R_\sigma [G]$ is left $(n,d)$-coherent if and only if $R_\sigma [H]$ is left $(n,d)$-coherent. These assertions follow by the same arguments as in \cite[Prop. 3.4]{GP25}, together with \cite[Lems. 2.2 and 3.1]{GP25}. For $d\ge \gD(R)$, this recovers \cite[Prop. 3.4]{GP25}.
\item Since $\mathsf{M}_{k\times k}(R)$ is a separable extension of
$R$ for every $k\in\mathbb{N}^*$, part (b) implies that $R$ is left
$(n,d)$-coherent if and only if $\mathsf{M}_{k\times k}(R)$ is left
$(n,d)$-coherent.
\end{enumerate}
\end{rmk}

We conclude this section with the following observation. Set $\mathcal{C} = \mathsf{FP}_1^{<\infty}(R)$ and consider the cotorsion pair $(\mathcal{A}, \mathcal{B})$ cogenerated by $\mathcal{C}$, where $\mathcal{B} = \mathcal{C}^{\perp_1}$ and $\mathcal{A} = {}^{\perp_1}\mathcal{B}$. By \cite[Lem. 2.1]{AT02}, if $R$ is left coherent, every module in $\mathcal{A}$ is a direct limit of modules in $\mathcal{C}$. We do not know whether this assertion remains valid under the weaker hypothesis that $R$ is left $(1,d)$-coherent for every $d\in\mathbb{N}^*$.

\begin{question}
Let $R$ be a ring that is left $(1,d)$-coherent for every $d\in\mathbb{N}^*$, and set $\mathcal{C}=\mathsf{FP}_1^{<\infty}(R)$. Is every module in $\mathcal{A}={}^{\perp_1}(\mathcal{C}^{\perp_1})$ a direct limit of modules in $\mathcal{C}$?
\end{question}

Note that if $R$ is left \newterm{1-regular} (i.e., every finitely
$1$-presented $R$-module has finite projective dimension), then the answer
to the Question is affirmative, as it reduces to \cite[Lem. 2.1]{AT02}.

\section{$(n,d)$-coherence, regularity, and $\mathsf{K}$-theory}\label{s: k}

We refer to \cite[Ch. 3]{Rosenberg} for the background on algebraic $\mathsf{K}$-theory. It is
well known that the Grothendieck group of $\mathsf{proj}(R)$ is denoted simply by $\mathsf{K}_0(R)$,
which is called the Grothendieck group of $R$. Since $\mathsf{K}_0(R)$ is
defined in terms of finitely generated projective $R$-modules, it is
natural to ask when $\mathsf{FP}_n^{\le d}(R)$ coincides with $\mathsf{proj}(R)$. Let $n\in \mathbb{N}^*$ and $d\in\mathbb{N}^*\cup\{\infty\}$, if $M\in\mathsf{FP}_n^{\leq d}(R)$, then $M$ is projective if and only if $\Ext_R^1(M,N)=0$ for every $N\in\mathsf{FP}_{n-1}^{d-1}(R)$. When $d=\infty$, this recovers \cite[Lem. 7.11(2)]{PA25}. Motivated by the proof of \cite[Thm. 2.1]{M01}, we now characterize when $\mathsf{FP}_n^{\leq d}(R)=\mathsf{proj}(R)$.

\begin{theorem}\label{theo: mah}
Let $R$ be a ring, let $n\in \mathbb{N}^*$ and $d\in \mathbb{N}^*$. Then
$\mathsf{FP}_n^{\le d}(R)=\mathsf{proj}(R)$ if and only if $R$ is left $(n,d)$-coherent and the left annihilator of every finitely generated proper ideal of $R^{\op}$ is nonzero.
\end{theorem}
\begin{proof}
Assume first that $R$ is left $(n,d)$-coherent and satisfies the annihilator condition.
Let $M \in \mathsf{FP}_n^{\le d}(R)$. Then $M$ admits a finite projective resolution
$ P_d \rightarrowtail P_{d-1} \xrightarrow{f_{d-1}}
\cdots \xrightarrow{f_1} P_0 \overset{f_0}\twoheadrightarrow M,$
where each $P_i$ is a finitely generated projective $R$-module. Since $P_d = \ker(f_{d-1})$ is a projective submodule of the projective module $P_{d-1}$, it is a direct summand of $P_{d-1}$ by \cite[Thm. 5.4]{B60}.
Consider the short exact sequence $P_d \rightarrowtail P_{d-1} \twoheadrightarrow \Img(f_{d-1})$. It follows that $\Img(f_{d-1})$ is a finitely generated projective $R$-module. Iterating this argument, we conclude that $M = \Img(f_0)$ is a finitely generated projective $R$-module.
Conversely, assume that $\mathsf{FP}_n^{\le d}(R)=\mathsf{proj}(R)$. Then $R$ is trivially
left $(n,d)$-coherent. Let $P \subseteq Q$ be projective $R$-modules with $P$ finitely generated. We show that
$Q/P$ is projective. We may assume that $Q$ is free.  If $Q$ is finitely generated, then
$Q/P\in\mathsf{FP}_n^{\leq d}(R)$, and hence $Q/P$ is projective. If $Q$ is not finitely generated, then $Q=Q_1\oplus Q_2$ for free $R$-modules $Q_1$ and $Q_2$, where $Q_1$ is finitely generated and contains $P$. As above, $Q_1=P\oplus P'$ for some
$R$-module $P'$, and thus $Q=P\oplus P'\oplus Q_2.$ Therefore $Q/P$ is projective. By \cite[Thm. 5.4]{B60}, the annihilator condition follows.
\end{proof}

Further characterizations of this equality are given in Proposition \ref{prop: von neumman}. For the ring in
Example \ref{ex: 1}, we have $\mathsf{FP}_1^{\le d}(R)=\mathsf{proj}(R)$ for every $d\in\mathbb{N}^*$. Hence, Theorem \ref{theo: mah} shows that the annihilator of every finitely generated proper ideal of $R$ is nonzero.\\

We now turn to the consequences of $(n,d)$-coherence in $\mathsf{K}$-theory. We shall use the following version of Quillen's Resolution Theorem.

\begin{theorem}[\textbf{Resolution Theorem}]\label{thm: res}
Let $\mathcal{M}$ be an exact category and let $\mathcal{P}$ be a full
subcategory of $\mathcal{M}$ closed under extensions. Assume that:
\begin{itemize}
\item If $M\rightarrowtail P\twoheadrightarrow P'$ is exact in
$\mathcal{M}$ with $P,P'\in\mathcal{P}$, then $M\in\mathcal{P}$.
\item Every object of $\mathcal{M}$ admits a finite resolution by
objects of $\mathcal{P}$.
\end{itemize}
Then the inclusion $\mathcal{P}\hookrightarrow\mathcal{M}$ induces
isomorphisms $\mathsf{K}_i(\mathcal{P})\xrightarrow{\ \cong\ }\mathsf{K}_i(\mathcal{M})$ for all $ i\ge 0.$
\end{theorem}

Recall that a ring $R$ is left \newterm{$n$-regular} if $\mathsf{FP}_n(R)=\mathsf{FP}_n^{<\infty}(R)$; see \cite{EP22}. As a consequence of Quillen's Resolution Theorem, if $R$ is both left $n$-coherent and left $n$-regular, then $\mathsf{K}_i(R)\cong \mathsf{K}_i(\mathsf{FP}_n(R))$ for all $i\ge 0$; see \cite{EP22}. Moreover, when $R$ is commutative, such a ring is $\mathsf{K}_0$-regular \cite[Cor. 8.2]{PA25}. The next proposition shows that, for $n$-regular rings, $n$-coherence is equivalent to $(n,d)$-coherence for every $d\in\mathbb{N}^*$. We omit the straightforward proof.

\begin{prop}
Let $R$ be left $n$-regular, where $n\in\mathbb{N}^*$. Then $R$ is left $n$-coherent if and only if it is left $(n,d)$-coherent for every $d\in\mathbb{N}^*$.
\end{prop}

As an application of $(n,d)$-coherence to algebraic $\mathsf{K}$-theory, we obtain the following result.

\begin{prop}\label{prop: k}
Let $R$ be a left $(n,d)$-coherent ring, with $n,d\in\mathbb{N}^*$. Then $\mathsf{K}_i(R)\cong
\mathsf{K}_i(\mathsf{FP}_n^{\le d}(R))$ for all $i\ge 0$.
\end{prop}
\begin{proof}
This follows from Proposition \ref{prop:equiv} and Theorem \ref{thm: res} applied to the inclusion
$\mathsf{proj}(R)\hookrightarrow\mathsf{FP}_n^{\leq d}(R)$.
\end{proof}

The following corollary is an immediate consequence of Remark \ref{rmk: 1}(a) and the previous proposition.

\begin{cor}
Let $R$ be a ring with $\wD(R)\le n<\infty$. Then
$\mathsf{K}_i(R)\cong\mathsf{K}_i(\mathsf{FP}_{n+1}^{\le d}(R))$ for all
$i\ge 0$ and every $d\in\mathbb{N}^*$. Moreover, $\mathsf{K}_i(R[t])\cong
\mathsf{K}_i(\mathsf{FP}_{n+2}^{\le d}(R[t]))$ for all $i\ge 0$ and every $d\in\mathbb{N}^{*}$.
\end{cor}

Coherent regular rings (i.e., rings which are both left $1$-regular and left $1$-coherent) form a natural class in algebraic $\mathsf{K}$-theory and satisfy several remarkable properties. For instance, $\mathsf{K}_{-1}(R)=0$ by \cite[Thm. 3.30]{AGH19}, and $\mathsf{K}_i(R)\cong\mathsf{K}_i(R[t_1,\cdots ,t_k])$ for $i\ge k-1$ by \cite[Cor. 5.3]{BL23}. Since coherent regularity can be difficult to verify, in the remainder of this section we will focus on the case $n=1$.

\begin{theorem}\label{thm: coh regular} 
Let $R$ be a ring, and let $\mathsf{CFP}_{\infty}(R)$ denote the class of all cyclic
$R$-modules of type $\mathsf{FP}_{\infty}(R)$. Then $R$ is left coherent if and only if every cyclic $R$-module is a direct limit of modules in $\mathsf{CFP}_{\infty}(R)$.
\end{theorem}
\begin{proof} 
Assume that every cyclic $R$-module is a direct limit of cyclic modules in $\mathsf{FP}_{\infty}(R)$, and let $I$ be a finitely generated ideal of $R$. Since $R/I$ is cyclic, there exists a directed system $(M_\lambda)_{\lambda\in\Lambda}$ of modules in
$\mathsf{CFP}_{\infty}(R)$ such that $R/I\cong\varinjlim_{\lambda\in\Lambda}M_\lambda$. As $R/I$ is finitely presented, the natural isomorphism $\Hom_R(R/I,\varinjlim_{\lambda\in\Lambda}M_\lambda)\cong \varinjlim_{\lambda\in\Lambda}\Hom_R(R/I,M_\lambda)$ shows that the identity map of $R/I$ factors through some $M_\lambda$. Hence $R/I$ is a direct summand of $M_\lambda$, and therefore is of type $\mathsf{FP}_{\infty}(R)$. Consequently, $I$ is finitely presented. Conversely, assume that $R$ is left coherent. Let $M\cong R/I$ be a cyclic $R$-module. Write $I=\varinjlim_{\lambda\in\Lambda}I_\lambda$, where each $I_\lambda$ is a finitely generated ideal of $R$ and $I_\lambda\subseteq I$. Since filtered direct limits are exact, passing to direct limits in the short exact sequences $I_\lambda\rightarrowtail R\twoheadrightarrow R/I_\lambda$ yields the following commutative diagram with exact rows: 
$$\begin{tikzcd}
 &
\varinjlim\limits_{\lambda} I_\lambda
\arrow[r, tail]
\arrow[d, "\cong"'] &
\varinjlim\limits_{\lambda} R
\arrow[r, two heads]
\arrow[d, "\cong"'] &
\varinjlim\limits_{\lambda} R/I_\lambda
\arrow[d, "\cong"] &\\
 & I\arrow[r, tail] & R\arrow[r, two heads] & R/I. & 
\end{tikzcd}$$

It follows that $R/I\cong\varinjlim_{\lambda\in\Lambda}R/I_\lambda$. Since each $I_\lambda$ is finitely generated and $R$ is left coherent, every $R/I_\lambda$ is finitely presented, hence of type $\mathsf{FP}_{\infty}(R)$. 
\end{proof}

Recall that a left coherent ring $R$ is left regular if and only if every finitely generated ideal of $R$ has finite projective dimension.

\begin{cor}\label{cor: coh regular}
Let $R$ be a ring. Then every cyclic $R$-module is a direct limit of modules in $\mathsf{CFP}_\infty^{<\infty}(R)$ if and only if $R$ is left coherent regular.
\end{cor}
\begin{proof}  By Theorem \ref{thm: coh regular}, it remains to verify the regularity
condition. Suppose that every cyclic $R$-module is a direct limit of modules in $\mathsf{CFP}_{\infty}^{<\infty}(R)$. Let $I$ be a finitely generated ideal of $R$. Then there exists a directed system
$(M_\lambda)_{\lambda\in\Lambda}$ in
$\mathsf{CFP}_{\infty}^{<\infty}(R)$ such that
$R/I\cong\varinjlim_{\lambda\in\Lambda}M_\lambda$. Since $R/I$ is finitely
presented, $\id_{R/I}$ factors through some $M_\lambda$. Hence $R/I$ is
a direct summand of $M_\lambda$, and therefore has finite projective
dimension. Consequently, $I$ has finite projective dimension. Thus $R$
is left regular. The converse is immediate.
\end{proof}

\begin{rmk}
Using \cite[Prop. 2.4]{BGH14} and arguing as above, one also obtains that every $R$-module is a direct limit of modules in $\mathsf{FP}_\infty^{<\infty}(R)$ if and only if $R$ is left coherent regular.
\end{rmk}

A special class of coherent regular rings consists of those coherent rings $R$ 
for which the projective dimension of the class of finitely presented 
$R$-modules is uniformly bounded. That is, there exists $d\in\mathbb{N}$ 
such that $\mathsf{FP}_1(R)=\mathsf{FP}_1^{\le d}(R).$ Examples of such rings include coherent rings of finite weak dimension. These rings will be called left \newterm{coherent uniformly $d$-regular}; see \cite{PA25}. Recall that an $R$-module $P$ is \newterm{$\mathsf{FP}_1$-projective} if $\Ext^1_R(P,M)=0$ for every $\mathsf{FP}_1$-injective (or absolutely pure) $R$-module $M$, and we denote by $\mathsf{FP}_1\text{-}\mathsf{Proj}(R)$ the class of all $\mathsf{FP}_1$-projective $R$-modules. We now characterize left coherent uniformly $d$-regular rings in terms of $\mathcal{C}$-periodic modules\footnote{An $R$-module $M$ is \newterm{$\mathcal{C}$-periodic} if there exists a short exact sequence $M\rightarrowtail C\twoheadrightarrow M$ with $C\in\mathcal{C}$.}.

\begin{prop}
Let $R$ be a ring and let $d\in\mathbb{N}$. Then $R$ is left coherent uniformly $d$-regular if and only if every $\mathsf{FP}_1\text{-}\mathsf{Proj}(R)$-periodic $R$-module is $\mathsf{FP}_1$-projective and has projective dimension at most $d$.
\end{prop}
\begin{proof}
If $R$ is left coherent uniformly $d$-regular, then every $\mathsf{FP}_1\text{-}\mathsf{Proj}(R)$-periodic $R$-module is $\mathsf{FP}_1$-projective by \cite[Cor. 4.10]{BHP24}, and thus has projective dimension at most $d$ by \cite[Prop. 9.4]{EFHO25}.
Conversely, left coherence of $R$ follows again from \cite[Cor. 4.10]{BHP24}. For any $M\in\mathsf{FP}_1(R)$, the split sequence $M\rightarrowtail M\oplus M\twoheadrightarrow M$ shows that $M$ is $\mathsf{FP}_1\text{-}\mathsf{Proj}(R)$-periodic, so $\pd_R(M)\le d$.
\end{proof}

\begin{rmk} Let $R$ be a ring.
\begin{enumerate}\renewcommand{\labelenumi}{(\alph{enumi})}
\item By \cite[Cor. 2.13]{ET11}, if an $R$-module $M$ satisfies $\fd_R(M)=\pd_R(M)<\infty$, then $\pd_R(M)\leq\id_R(R)$. Since $\fd_R(M)=\pd_R(M)$ for every $M\in\mathsf{FP}_\infty(R)$, it follows that, for a left coherent ring $R$ with  $d=\id_R(R)<\infty$, $R$ is left regular if and only if it is left uniformly $d$-regular. 
\item If the polynomial ring $R[t]$ is left coherent and $\id_R(R[t])<\infty$, then $R[t]$ is regular if and only if it is uniformly $d$-regular for some $d\in\mathbb{N}$. This follows from part (a) together with \cite[Cor. 2.2]{EEI08}.
\end{enumerate}
\end{rmk}


\section{$\mathsf{FP}_n^{\le d}$-injective and $\mathsf{FP}_n^{\le d}$-projective modules}\label{s:inj and proj relativos} 
Following Lee \cite{Lee02}, an $R$-module $E$ is said to be \newterm{$d$-absolutely pure}, for $d\in\mathbb{N}^*\cup\{\infty\}$, 
if $\Ext_R^1(K,E)=0$ for every $K\in\mathsf{FP}_1^{\le d}(R)$. Motivated by this concept, we introduce the following generalization.

\begin{dfn}
Let $R$ be a ring, and let $n\in \mathbb{N}^*$ and $d\in\mathbb{N}^*\cup\{\infty\}$. An $R$-module $E$ is said to be 
\newterm{$\mathsf{FP}_n^{\le d}$-injective} if $\Ext_R^1(K,E)=0$ for every $K\in\mathsf{FP}_n^{\le d}(R)$.
\end{dfn}

Denote by $\mathsf{FP}_n^{\le d}\text{-}\mathsf{Inj}(R)$ the class of all $\mathsf{FP}_n^{\le d}$-injective $R$-modules. For any $n,d\in\mathbb{N}^*$, we have the inclusion $\mathsf{FP}_n\text{-}\mathsf{Inj}(R)\subseteq 
\mathsf{FP}_n^{\le d}\text{-}\mathsf{Inj}(R)$. Here, $\mathsf{FP}_n\text{-}\mathsf{Inj}(R)$ denotes 
the class of $\mathsf{FP}_n$-injective $R$-modules introduced by Zhou in \cite{Zhou04} (that is, $R$-modules $E$ such that $\Ext_R^1(K,E)=0$ for all $K\in\mathsf{FP}_n(R)$). In particular, when $n=1$, the $\mathsf{FP}_1^{\le d}$-injective modules coincide with the $d$-absolutely pure modules.\\

\begin{rmk}~\
\begin{enumerate}\renewcommand{\labelenumi}{(\alph{enumi})}
\item Let $R$ be the ring from Example \ref{ex: 1}. Since $\mathsf{FP}_1^{\le d}(R)=\mathsf{proj}(R)$ for every $d\in \mathbb{N}^*$, it follows that $\mathsf{FP}_1^{\le d}\text{-}\mathsf{Inj}(R)=\lMod R$. On the other hand, \cite[Ex. 5.7]{BP} shows that $\mathsf{FP}_1\text{-}\mathsf{Inj}(R)\subsetneq \lMod R$. Therefore, $\mathsf{FP}_1\text{-}\mathsf{Inj}(R)\subsetneq \mathsf{FP}_1^{\le d}\text{-}\mathsf{Inj}(R)$.
\item Following \cite[Def. 3.1]{GD09}, an $R$-module $M$ is called
\newterm{$\mathsf{P}^{\le d}$-injective}, for $d \in \mathbb{N}^*$ if
$\Ext_R^1(K,M)=0$ for every $R$-module $K$ with $\pd_R(K)\le d$.
Since $\mathsf{FP}_n^{\le d}(R)\subseteq \mathsf{P}^{\le d}(R)$, every
$\mathsf{P}^{\le d}$-injective $R$-module is $\mathsf{FP}_n^{\le d}$-injective. In particular, by \cite[Rmks. 3.2 and 3.3]{GD09}, Gorenstein injective modules and $d$-th cosyzygies of $R$-modules are examples of $\mathsf{FP}_n^{\leq d}$-injective modules.
\end{enumerate}
\end{rmk}

In \cite{MD06}, Mao and Ding defined a short exact sequence of $R$-modules $\mathfrak{S}:$
$A\rightarrowtail B \twoheadrightarrow C$ to be \newterm{$n$-pure} if the induced sequence $\Hom_R(M,B)\twoheadrightarrow \Hom_R(M,C)$
is exact for every finitely $n$-presented $R$-module $M$.

\begin{dfn}
Let $R$ be a ring, $n\in\mathbb{N}^*$ and $d\in\mathbb{N}^*\cup\{\infty\}$.  
A short exact sequence $A\rightarrowtail B\twoheadrightarrow C$
of $R$-modules is said to be \newterm{$\mathsf{FP}_n^{\le d}$-pure} if the induced sequence 
$\Hom_R(M,B)\twoheadrightarrow \Hom_R(M,C)$ is exact for every $M\in\mathsf{FP}_n^{\le d}(R)$
\end{dfn}

In particular, when $n=1$, this notion coincides with the notion of purity introduced by Lee in \cite{Lee03}. The following characterization is immediate.

\begin{prop} Let $R$ be a ring, and let $n\in\mathbb{N}^*$ and $d\in\mathbb{N}^*\cup\{\infty\}$. 
For an $R$-module $E$, the following statements are equivalent:
\begin{enumerate}\renewcommand{\labelenumi}{(\alph{enumi})}
\item $E$ is $\mathsf{FP}_n^{\le d}$-injective.
\item Every exact sequence $E \rightarrowtail E_1 \twoheadrightarrow E_2$ of  $R$-modules is  $\mathsf{FP}_n^{\le d}$-pure.
\item There exists an $\mathsf{FP}_n^{\le d}$-pure exact sequence 
$ E \rightarrowtail E_1 \twoheadrightarrow E_2 $ of  $R$-modules with $E_1$ injective.
\item There exists an $\mathsf{FP}_n^{\le d}$-pure exact sequence 
$E \rightarrowtail E_1 \twoheadrightarrow E_2$ of $R$-modules with $E_1$ $\mathsf{FP}_n^{\le d}$-injective.
\end{enumerate}
\end{prop}

We next prove a lifting result for $\mathsf{FP}_n^{\le d}$-injective modules, using arguments analogous to those in \cite[Prop. 2.3]{XK19}. When $d\ge\gD(R)$, this recovers and extends \cite[Prop. 4.4]{BP19}.

\begin{prop}\label{prop:FPn-d-injective}
Let $R$ be a ring, let $n\in\mathbb{N}^*$, and let
$d\in\mathbb{N}^* \cup\{\infty\}$. For an $R$-module $E$, the following
statements are equivalent:
\begin{enumerate}\renewcommand{\labelenumi}{(\alph{enumi})}
\item $E$ is $\mathsf{FP}_n^{\le d}$-injective.
\item For every diagram
$$\begin{tikzcd}
N \arrow[r, tail, "g"] \arrow[d,"f"'] & P \arrow[dl,dashed,"h"]\\
E &
\end{tikzcd}$$
where $P$ is a projective $R$-module and $N \subseteq P$ is a submodule with
$N \in \mathsf{FP}_{n-1}^{\le d-1}(R)$, there exists a homomorphism
$h \colon P \to E$ such that $h \circ g = f$.
\item For every diagram
$$\begin{tikzcd}
N \arrow[r, tail, "g"] \arrow[d,"f"'] & P \arrow[dl,dashed,"h"]\\
E &
\end{tikzcd}$$
where $P$ is a finitely generated projective $R$-module and $N\subseteq P$ is a submodule with
$N \in \mathsf{FP}_{n-1}^{\le d-1}(R)$, there exists a homomorphism
$h \colon P\to E$ such that $h\circ g=f$.
\end{enumerate}
\end{prop}
\begin{proof} ~\
\begin{itemize}
\item (a)$\Rightarrow$(b): Assume that $E$ is $\mathsf{FP}_n^{\le d}$-injective, and let $N$ be a submodule of a projective $R$-module $P$ such that $N$ is finitely $(n-1)$-presented and  $\pd_R(N) \le d-1$. We first consider the case where $P$ is a free $R$-module. Since $N$ is finitely $(n-1)$-presented, it is, in particular, finitely generated. As $N$ is a submodule of the free module $P$, there exist free $R$-modules $F_1$ and $F_2$, with $F_1$ finitely generated, such that $P\cong F_1 \oplus F_2$ and $N \subseteq F_1$. Then $F_1/N$ is finitely $n$-presented with $\pd_R(F_1/N)\le d$.  
Therefore $\Ext^1_R(F_1/N, E)=0$. It follows that, for any homomorphism $f: N\to E$, there exists a homomorphism 
$\beta_0: F_1 \to E$  such that the following diagram commutes:
$$\begin{tikzcd}[row sep=huge, column sep=large]
 & N \arrow[r, tail, "g"] \arrow[d, equal] & P \arrow[r, two heads] \arrow[d, "p"] & P/N  \arrow[d] &  \\
& N \arrow[r, tail, "g"]\arrow[d, "f"'] & F_1 \arrow[r, two heads] \arrow[dl, dashed, "\beta_0"'] & F_1/N  &  \\
& E & & &
\end{tikzcd}$$
Consequently, there exists $h=\beta_0 \circ p: P\to E$ such that $f =h\circ g$. Hence  $\Ext^1_R(P/N, E) = 0.$ Now let $P$ be a projective module.  Then there exist a free module $F_0$ and a projective module $Q$ such that $F_0 \cong P \oplus Q$. Set $C:= P/N$ and $C_0:= F_0/N$. Consider the following commutative diagram with exact rows:
$$\begin{tikzcd}
 & N \arrow[r, tail, ""]\arrow[d, equal] & P \arrow[r, two heads] \arrow[d, tail, ""] & C  \arrow[d, tail, ""] &  \\
 & N \arrow[r, tail, ""] & F_0 \arrow[r, two heads] \arrow[d, two heads] & C_0  \arrow[d, two heads] &  \\
 & & Q \arrow[r, equal]  & Q  & 
\end{tikzcd}$$
Since $Q$ is projective, $C_0\cong C\oplus Q$. By the previous case, we have 
$\Ext^1_R(C_0,E)=0$, and hence $\Ext^1_R(C,E)=0.$\\
\item Finally, (b)$\Rightarrow$(c) and (c)$\Rightarrow$(a) are immediate.
\end{itemize}
\end{proof}

\begin{rmk}\label{rmk: complete} 
Let $R$ be a ring, $n\in\mathbb{N}^*$ and $d\in\mathbb{N}^*\cup\{\infty\}$.
\begin{enumerate}\renewcommand{\labelenumi}{(\alph{enumi})}
\item An $R$-module $E$ is $\mathsf{FP}_n^{\leq d}$-injective if and only if
$\Ext_R^1(P/N,E)=0$ for every projective $R$-module $P$ and every
submodule $N\subseteq P$ with $N\in\mathsf{FP}_{n-1}^{\leq d-1}(R)$.
\item Let $\mathcal{C}$ be a set of representatives of all finitely $n$-presented $R$-modules of projective dimension at most $d$. Then
$\mathcal{C}^{\perp_1}=\mathsf{FP}_n^{\le d}\text{-}\mathsf{Inj}(R)$.
By the Eklof-Trlifaj Theorem \cite{ET01}, the cotorsion pair
$\bigl({}^{\perp_1}(\mathsf{FP}_n^{\le d}\text{-}\mathsf{Inj}(R)),\;
\mathsf{FP}_n^{\le d}\text{-}\mathsf{Inj}(R)\bigr)$
is \newterm{complete}.  An $R$-module is called \newterm{$\mathsf{FP}_n^{\le d}$-projective} if it belongs to
${}^{\perp_1}(\mathsf{FP}_n^{\le d}\text{-}\mathsf{Inj}(R))$, and the class of all such
modules is denoted by $\mathsf{FP}_n^{\le d}\text{-}\mathsf{Proj}(R)$. Consequently,
every $R$-module admits a monomorphism into an $\mathsf{FP}_n^{\le d}$-injective module
with $\mathsf{FP}_n^{\le d}$-projective cokernel, and every module is a quotient of an
$\mathsf{FP}_n^{\le d}$-projective module by an $\mathsf{FP}_n^{\le d}$-injective
submodule. According to \cite[Cor. 3.2.4]{GT12}, $\mathsf{FP}_n^{\le d}\text{-}\mathsf{Proj}(R)$ consists of all direct summands of $\mathcal{C}$-filtered modules. Moreover, the class $\mathsf{FP}_n^{\le d}\text{-}\mathsf{Proj}(R)$ is closed under extensions, direct sums, direct summands, and filtrations, while
$\mathsf{FP}_n^{\le d}\text{-}\mathsf{Inj}(R)$ is closed under extensions, direct
products, and direct summands. If $n \ge 2$, the latter class is also closed under direct
limits; see \cite[Lem 2.9.2]{CD96}, and in particular under coproducts.
\item Fix an integer $n$. For each $d\in\mathbb{N}^*$, consider the family of cotorsion pairs
$$\bigl(\mathsf{FP}_n^{\le d}\text{-}\mathsf{Proj}(R),
\mathsf{FP}_n^{\le d}\text{-}\mathsf{Inj}(R)\bigr)_{d \in \mathbb{N}^*}.$$
By the lattice structure of cotorsion pairs; see \cite{GSW01}, the infimum of this family is
$$\left(\bigcap_{d\in\mathbb{N}^*} \mathsf{FP}_n^{\le d}\text{-}\mathsf{Proj}(R),
\left(\bigcap_{d\in\mathbb{N}^*} \mathsf{FP}_n^{\le d}\text{-}\mathsf{Proj}(R)\right)^{\perp_1}
\right),$$
which is generated by the class
$\bigcup_{d \in \mathbb{N}^*} \mathsf{FP}_n^{\le d}\text{-}\mathsf{Inj}(R).$
Furthermore, the supremum of this family is given by
$$\left(
{}^{\perp_1}\!\left(\bigcap_{d\in\mathbb{N}^*}
\mathsf{FP}_n^{\le d}\text{-}\mathsf{Inj}(R)\right),
\bigcap_{d\in\mathbb{N}^*} \mathsf{FP}_n^{\le d}\text{-}\mathsf{Inj}(R)
\right).$$
\item If $R$ is commutative, then every $\mathsf{FP}_n^{\leq d}$-injective
$R$-module is divisible. In particular, if $R$ is an integral domain with quotient
field $K\neq R$, then $K/R$ is $\mathsf{FP}_n^{\leq d}$-injective if and only if
$\Ext_R^1(N,R)=0$ for every submodule $N$ of a projective $R$-module with
$N\in\mathsf{FP}_{n-1}^{\leq d-1}(R)$. The latter equivalence follows by an argument analogous to
\cite[Thm. 2.3]{XK19}.
\item By \cite[Lem. 3.5]{ODL14}, if $R$ is commutative, $n\geq2$, and $M$ is an $R$-module, then $M\in \mathsf{FP}_n^{\le d}\text{-}\mathsf{Inj}(R)$ if $M_{\mathfrak p}\in\mathsf{FP}_n^{\le d}\text{-}\mathsf{Inj}(R_{\mathfrak p})$ for every $\mathfrak p\in\mathsf{Spec}(R)$.
\end{enumerate}
\end{rmk}

\begin{prop}
Let $R$ be a ring, $n\in\mathbb{N}^*$ and $d\in\mathbb{N}^*\cup\{\infty\}$.
For an $R$-module $M$, the following statements are equivalent:
\begin{enumerate}\renewcommand{\labelenumi}{(\alph{enumi})}
\item $M$ is $\mathsf{FP}_n^{\le d}$-projective.
\item $M$ is projective with respect to every short exact sequence
$A\rightarrowtail B\twoheadrightarrow C$ of $R$-modules with
$A\in\mathsf{FP}_n^{\le d}\text{-}\mathsf{Inj}(R)$.
\item $M$ is projective with respect to every short exact sequence
$A\rightarrowtail E(A) \overset{\pi}\twoheadrightarrow E(A)/A$, where $A \in \mathsf{FP}_n^{\le d}\text{-}\mathsf{Inj}(R)$ and $E(A)$ denotes the injective envelope of $A$.
\end{enumerate}
\end{prop}
\begin{proof} The proof is straightforward.
\end{proof}

When $d\ge \gD(R)$, we recover \cite[Prop. 2.2]{Zhu17}.

\begin{lemma}
Let $R$ be a ring, $n\in\mathbb{N}^*$, and $d\in \mathbb{N}^* \cup \{\infty\}$. Consider a short exact sequence of $R$-modules $A\rightarrowtail B\twoheadrightarrow C$.Then the following hold:
\begin{enumerate}\renewcommand{\labelenumi}{(\alph{enumi})}
\item If $A$ is $\mathsf{FP}_n^{\le d-1}$-injective and $B$ is $\mathsf{FP}_{n+1}^{\le d}$-injective, then $C$ is $\mathsf{FP}_{n+1}^{\le d}$-injective.
\item If $B$ is $\mathsf{FP}_n^{\le d}$-projective and $C$ is $\mathsf{FP}_{n+1}^{\le d}$-projective, then $A$ is $\mathsf{FP}_n^{\le d}$-projective.
\end{enumerate}
\end{lemma}
\begin{proof} ~\
\begin{itemize}
\item  For part \textbf{(a)}, let $M\in\mathsf{FP}_{n+1}^{\le d}$. There exists an exact sequence 
$K\rightarrowtail F\twoheadrightarrow M,$ where $F$ is finitely generated and free, and $K\in\mathsf{FP}_n^{\le d-1}(R)$. Since $A$ is $\mathsf{FP}_n^{\le d-1}$-injective, $\Ext^1_R(K,A)=0$, and hence $\Ext^2_R(M,A)=0$. Applying $\Hom_R(M,-)$ to the given short exact sequence gives $0=\Ext^1_R(M,B)\to\Ext^1_R(M,C)\to\Ext^2_R(M,A)=0$. Thus, $\Ext^1_R(M,C)=0$, and so $C$ is $\mathsf{FP}_{n+1}^{\le d}$-injective.
\item For part \textbf{(b)}, let $N\in \mathsf{FP}_n^{\le d}\text{-}\mathsf{Inj}(R)$. Applying $\Hom_R(-,N)$ to $A\rightarrowtail B\twoheadrightarrow C$, we have $0=\Ext^1_R(B,N)\to \Ext^1_R(A,N)\to \Ext^2_R(C,N)$. Let $N\rightarrowtail E \twoheadrightarrow E/N$ be an exact sequence with $E$ injective. Since $\mathsf{FP}_n^{\le d}\text{-}\mathsf{Inj}(R) \subseteq \mathsf{FP}_n^{\le d-1}\text{-}\mathsf{Inj}(R)$, it follows from part (a) that $E/N$ is $\mathsf{FP}_{n+1}^{\le d}$-injective, so
$\Ext^2_R(C,N)\cong\Ext^1_R(C,E/N)=0$. Thus, $\Ext^1_R(A,N)=0$, and hence $A$ is $\mathsf{FP}_n^{\le d}$-projective.
\end{itemize}
\end{proof}

\begin{prop}\label{prop: fpn-proj acotada dim}
Let $R$ be a ring, and let $n,d\in\mathbb{N}^*$. 
Then every $\mathsf{FP}_n^{\le d}$-projective $R$-module has projective 
dimension at most $d$.
\end{prop}
\begin{proof}
The result follows from \cite[Prop. 3]{A55}, noting that every 
$\mathsf{FP}_n^{\le d}$-projective $R$-module $M$ admits a filtration 
by finitely $n$-presented $R$-modules of projective dimension at most $d$.
\end{proof}

When $d\ge \gD(R)$, the following result recovers the equivalence between
conditions (1) and (2) in \cite[Thm. 2.1]{ODL14}.

\begin{cor}\label{cor: rel fpn y fpn proj}
Let $R$ be a ring, $n\in\mathbb{N}^*$ and $d\in\mathbb{N}^*\cup\{\infty\}$. If $M$ is a finitely generated $R$-module, then $M\in\mathsf{FP}_n^{\le d}\text{-}\mathsf{Proj}(R)$ if and only if $M\in\mathsf{FP}_n^{\le d}(R)$.    
\end{cor}
\begin{proof}
This follows from Proposition \ref{prop: fpn-proj acotada dim} and \cite[Thm. 2.1]{ODL14}.
\end{proof}


\subsection{$\mathsf{FP}_n^{\le d}$-injective and $\mathsf{FP}_n^{\le d}$-projective modules over $(n,d)$-coherent rings}  
We begin by characterizing left $(n,d)$-coherent rings in terms of $\mathsf{FP}_{n+1}^{\le d}$-injective modules.

\begin{theorem}
Let $R$ be a ring, $n\in\mathbb{N}^*$ and $d\in\mathbb{N}^*\cup\{\infty\}$. Then the following statements are equivalent:
\begin{enumerate}
\item[(1)] $R$ is left $(n,d)$-coherent.
\item[(15)] $\Ext_R^1(M,N)=0$ for all $M\in \mathsf{FP}_n^{\le d}(R)$ and all $N\in \mathsf{FP}_{n+1}^{\le d}\text{-}\mathsf{Inj}(R)$.
\end{enumerate}
\end{theorem}
\begin{proof} ~\
\begin{itemize}
\item (15)$\Rightarrow$(1): Suppose that $\Ext_R^1(M,N)=0$ for all $M\in\mathsf{FP}_n^{\le d}(R)$ and all $N\in\mathsf{FP}_{n+1}^{\le d}\text{-}\mathsf{Inj}(R)$. Let $M\in\mathsf{FP}_n^{\le d}(R)$. Since 
$\mathsf{FP}_{n+1}\text{-}\mathsf{Inj}(R)\subseteq \mathsf{FP}_{n+1}^{\le d}\text{-}\mathsf{Inj}(R)$,
we also have $\Ext_R^1(M,N)=0$ for each $N\in \mathsf{FP}_{n+1}\text{-}\mathsf{Inj}(R)$.
Thus $M\in \mathsf{FP}_{n+1}\text{-}\mathsf{Proj}(R)$, and by \cite[Thm. 2.1]{ODL14} it follows that $M$ is finitely $(n+1)$-presented. Hence, $R$ is left $(n,d)$-coherent.
\item (1)$\Rightarrow$(15): Immediate.
\end{itemize}
\end{proof}

When $d\ge\gD(R)$, we recover \cite[Cor. 4.7]{BP19}.

\begin{rmk}\label{rmk: coh prom}
For $n\in\mathbb{N}^*$, by \cite[Prop. 3.3]{Zhou04}, a ring $R$ is left $n$-coherent
if and only if $\Ext_R^{n+1}(M,N)=0$ for every finitely $n$-presented $R$-module $M$
and every $\mathsf{FP}_1$-injective $R$-module $N$. Consequently, if the class of $\mathsf{FP}_1$-injective modules has injective dimension at most $n$, then $R$ is left $n$-coherent. In our context, if $\Ext_R^{n+1}(M,N)=0$ for all $M\in\mathsf{FP}_n^{\le d}(R)$ and all $N\in\mathsf{FP}_1^{\le d}\text{-}\mathsf{Inj}(R)$, then $R$ is left $(n,d)$-coherent. However, whether the converse holds remains an open question: does left $(n,d)$-coherence of $R$ imply $\Ext_R^{n+1}(M,N)=0$ for all such $M$ and $N$?
\end{rmk}

We now prove the following theorem by adapting the arguments in \cite[Thm. 1]{Lee02}. 

\begin{theorem}\label{thm:n-coh 1}
Let $R$ be a ring, $n\in\mathbb{N}^*$ and $d\in\mathbb{N}^* \cup \{\infty\}$. Then the following statements are equivalent:
\begin{enumerate}
\item[(1)] $R$ is left $(n,d)$-coherent.
\item[(16)] $\Ext_R^2(M,N)=0$ for every finitely $n$-presented $R$-module $M$ with $\pd_R(M)\le d$ and every
$\mathsf{FP}_n^{\le d}$-injective $R$-module $N$.
\end{enumerate}
\end{theorem}
\begin{proof} ~\
\begin{itemize}
\item (1)$\Rightarrow$(16): Assume that $R$ is left $(n,d)$-coherent.
Let $M\in\mathsf{FP}_n^{\le d}(R)$ and let $N$ be $\mathsf{FP}_n^{\le d}$-injective.
By Proposition \ref{prop: car1}, $M$ is finitely $(n+1)$-presented, so there exists
a short exact sequence $H\rightarrowtail F\twoheadrightarrow M $, with $F$ finitely generated free and $H \in \mathsf{FP}_n^{\le d-1}(R)$. Since $N$ is $\mathsf{FP}_n^{\le d}$-injective, $\Ext_R^1(H,N)=0$, and hence $\Ext_R^2(M,N)=0$. 
\item (16)$\Rightarrow$(1): Conversely, assume that $\Ext_R^2(M,N)=0$ for all such $M$ and $N$. Let $M$ be a finitely $n$-presented $R$-module with $\pd_R(M)\le d$, and consider a short exact sequence
$K\rightarrowtail F\twoheadrightarrow M$, where $F$ is a finitely generated free $R$-module and $K\in\mathsf{FP}_{n-1}^{\le d}(R)$. Applying $\Hom_R(-,N)$ yields the exact sequence $0=\Ext_R^1(F,N)\to \Ext_R^1(K,N)\to\Ext_R^2(M,N)=0$, which implies $\Ext_R^1(K,N)=0$ for all $\mathsf{FP}_{n}^{\le d}$-injective $N$, and in particular for all $\mathsf{FP}_n$-injective $N$.  By \cite[Thm. 2.1]{ODL14}, it follows that $K$ is finitely $n$-presented. Consequently, $M$ is finitely $(n+1)$-presented, and by Proposition \ref{prop: car1}, it follows that $R$ is left $(n,d)$-coherent.
\end{itemize}
\end{proof}

When $n=1$, we recover \cite[Thm. 1]{Lee02}.

\begin{cor}\label{cor:n-d-coh-equiv 1}
Let $R$ be a ring,  $n\in\mathbb{N}^*$ and $d\in\mathbb{N}^*\cup \{\infty\}$.
Then the following statements are equivalent:
\begin{enumerate}
\item $R$ is left $(n,d)$-coherent.
\item[(17)] For every short exact sequence $ K \rightarrowtail P \twoheadrightarrow C $
with $C\in \mathsf{FP}_n^{\le d}(R)$ and $P$ projective,
the module $K$ is $\mathsf{FP}_n^{\le d}$-projective.
\item[(18)] $\Ext_R^{j+1}(C,N)=0$ for all $j \ge 0$,
all $C \in \mathsf{FP}_n^{\le d}(R)$, and all
$N \in \mathsf{FP}_n^{\le d}\text{-}\mathsf{Inj}(R)$.
\item[(19)] The class $\mathsf{FP}_n^{\le d}\text{-}\mathsf{Inj}(R)$
is closed under cokernels of monomorphisms.
\item[(20)] For every $N \in \mathsf{FP}_n^{<d}\text{-}\mathsf{Inj}(R)$,
the quotient $E(N)/N$ is again
$\mathsf{FP}_n^{<d}\text{-}\mathsf{Inj}(R)$.
\item[(21)] The cotorsion pair
$\big(
\mathsf{FP}_n^{\le d}\text{-}\mathsf{Proj}(R),\,
\mathsf{FP}_n^{\le d}\text{-}\mathsf{Inj}(R)
\big)$ 
is hereditary.
\item[(22)]  $\Ext_R^{j+1}(P,N)=0$ for all $j \ge 1$,
all $P \in \mathsf{FP}_n^{<d}\text{-}\mathsf{Proj}(R)$, and all
$N \in \mathsf{FP}_n^{<d}\text{-}\mathsf{Inj}(R)$.
\item[(23)] The class $\mathsf{FP}_n^{\le d}\text{-}\mathsf{Proj}(R)$
is closed under kernels of epimorphisms.
\item[(24)] Every $\mathsf{FP}_{n+1}^{\le d}$-injective $R$-module is
$\mathsf{FP}_{n}^{\le d}$-injective. 
\item[(25)] Every $\mathsf{FP}_n^{\le d}$-projective $R$-module is
$\mathsf{FP}_{n+1}^{\le d}$-projective.
\end{enumerate}
\end{cor}
\begin{proof} ~\
\begin{itemize}
\item (1)$\Leftrightarrow$(17)$\Leftrightarrow$(18)$\Leftrightarrow$(19)$\Leftrightarrow$(20)$\Leftrightarrow$(21)$\Leftrightarrow$(22): These equivalences follow from \cite[Thm.1]{Zhu2017} together with Theorem \ref{thm:n-coh 1}.
\item (22)$\Rightarrow$(23): Let $ A\rightarrowtail B\twoheadrightarrow C$ be an exact sequence of $R$-modules with $B,C\in \mathsf{FP}_n^{\le d}\text{-}\mathsf{Proj}(R)$. For every $\mathsf{FP}_n^{\le d}$-injective $R$-module $N$, applying $\Hom_R(-,N)$ yields the corresponding long exact sequence.  Since $B$ and $C$ are $\mathsf{FP}_n^{\le d}$-projective, we have $\Ext_R^1(B,N)=0$ and $\Ext_R^2(C,N)=0$, and hence $\Ext_R^1(A,N)=0$. Therefore, $A\in \mathsf{FP}_n^{\le d}\text{-}\mathsf{Proj}(R)$. 
\item (23)$\Rightarrow$(1): This follows from Proposition \ref{prop: car1} and Corollary \ref{cor: rel fpn y fpn proj}. 
\item (1)$\Rightarrow$(24)$\Rightarrow$(25): This is clear.
\item (25)$\Rightarrow$(1): Let $M\in \mathsf{FP}_{n}^{\le d}(R)$. Then $M$ belongs to $\mathsf{FP}_{n}^{\le d}\text{-}\mathsf{Proj}(R)$, and by assumption we have $M\in \mathsf{FP}_{n+1}^{\le d}\text{-}\mathsf{Proj}(R)$.
Hence, by Corollary \ref{cor: rel fpn y fpn proj},
$M\in \mathsf{FP}_{n+1}^{\le d}(R)$. Therefore, every finitely $n$-presented $R$-module of projective dimension
at most $d$ is finitely $(n+1)$-presented, and consequently, $R$ is left $(n,d)$-coherent.
\end{itemize}
\end{proof}

When $d\ge\gD(R)$, the above equivalences recover
the equivalence between conditions (1), (5), and (6)
in \cite[Thm. 4.1]{MD06}, as well as the equivalence between
conditions (1), (5), and (6) in \cite[Thm. 5.5]{BP}.

\begin{rmk} Let $R$ be a ring, $n\in\mathbb{N}^*$ and $d\in\mathbb{N}^*$. 
\begin{enumerate}\renewcommand{\labelenumi}{(\alph{enumi})}
\item Assume that $n\ge 2$. Recall that an $R$-module $T$ is called an \newterm{$d$-tilting module} if it has  projective dimension at most $d$, $\Ext_R^{i}(T,T^{(\lambda)})=0$ for every cardinal $\lambda$ and every $i>0$, and $R$ admits a finite coresolution  whose terms belong to $\Add (T)$, where $\Add (T)$ denotes the class of all direct  summands of direct sums of copies of $T$. The cotorsion pair generated by $T$ is called an \newterm{$d$-tilting cotorsion pair}. By \cite[Lem. 1.13]{PT}, a cotorsion pair 
$(\mathcal{A},\mathcal{B})$ is $d$-tilting if and only if $\mathcal{A}\subseteq \mathsf{P}^{\le d}(R)$, $\mathcal{B}$ is closed under  arbitrary direct sums, and $(\mathcal{A},\mathcal{B})$ is hereditary. Consequently, the cotorsion pair $(\mathsf{FP}_n^{\le d}\text{-}\mathsf{Proj}(R),\mathsf{FP}_n^{\le d}\text{-}\mathsf{Inj}(R))$ is $d$-tilting if and only if $R$ is a left $(n,d)$-coherent ring.
\item Following Bazzoni and Tarantino \cite[Prop. 3.2]{BT19} (to which we also 
refer for the corresponding notation), assume that $R$ is a left $(n,d)$-coherent 
ring. By Proposition \ref{prop: fpn-proj acotada dim}, for every 
$\mathsf{FP}_n^{\le d}$-projective $R$-module $M$ and every acyclic complex $Y$ 
with terms in $\mathsf{FP}_n^{\le d}\text{-}\mathsf{Inj}(R)$, the cycles 
$\mathsf{Z}^j(Y)$ belong to $M^{\perp}$. It follows that 
$Y\in \widehat{\mathsf{FP}_n^{\le d}\text{-}\mathsf{Inj}(R)}$,
and hence $\mathsf{ex}\,\mathsf{FP}_n^{\le d}\text{-}\mathsf{Inj}(R)=
\widehat{\mathsf{FP}_n^{\le d}\text{-}\mathsf{Inj}(R)}.$
In particular, in the abelian model structure associated with the cotorsion pair
$\big( \mathsf{FP}_n^{\le d}\text{-}\mathsf{Proj}(R),\mathsf{FP}_n^{\le d}\text{-}\mathsf{Inj}(R) \big)$; see \cite[Cor. 3.1]{BT19}, the class  $\mathsf{dw}\,\mathsf{FP}_n^{\le d}\text{-}\mathsf{Inj}(R)$ coincides with the 
class of fibrant objects. Moreover, if $n\ge 2$, then by item (a) 
and \cite[Prop. 3.3]{BT19} 
$$\mathsf{dw}\,\mathsf{FP}_n^{\le d}\text{-}\mathsf{Inj}(R)=
\mathsf{dg}\,\mathsf{FP}_n^{\le d}\text{-}\mathsf{Inj}(R).$$
Hence, there exists a model structure on $\mathsf{Ch}(R)$ in which the fibrant 
objects are precisely the complexes with components in the $d$-tilting class 
$\mathsf{FP}_n^{\le d}\text{-}\mathsf{Inj}(R)$ and the trivial objects are the 
acyclic complexes.
\end{enumerate}
\end{rmk}

Mao and Ding \cite[Def. 3.1]{MD07} introduced the notion of an 
\newterm{$(n,d)$-injective} module. Let $n\in\mathbb{N}^*$ and 
$d\in\mathbb{N}^*\cup\{\infty\}$. An $R$-module $M$ is said to be 
\newterm{$(n,d)$-injective} if $\Ext_R^n(K,M)=0$ for every finitely $n$-presented $R$-module $K$ with  $\pd_R(K)\le d$. The following proposition establishes the relationship between $(n,d)$-injective and $\mathsf{FP}_n^{\le d}$-injective modules. Moreover, when $d\ge\gD(R)$, the equivalence of conditions (1) and (2) in the proposition below is a consequence of \cite[Lem. 3.4(5)]{MD06}.

\begin{prop}
Let $R$ be a left $(n,d)$-coherent ring with $n\in\mathbb{N}^*$ and $d\in\mathbb{N}^*\cup\{\infty\}$. For an $R$-module $M$, the following statements are equivalent:
\begin{enumerate}\renewcommand{\labelenumi}{(\alph{enumi})}
\item $M$ is $(n,d)$-injective.
\item There exists an exact sequence $M\rightarrowtail E_0 \to E_1\to\cdots\twoheadrightarrow E_{n-1}$, where each $E_i$ is $\mathsf{FP}_n^{\le d}$-injective for all $0 \le i < n-1$.
\item $\Ext_R^{n}(Y,M)=0$ for every $\mathsf{FP}_n^{\le d}$-projective $Y$.
\end{enumerate}
\end{prop}
\begin{proof} ~\
\begin{itemize}
\item (b)$\Rightarrow$(a): Assume there exists an exact sequence as in (b). Since $R$ is left $(n,d)$-coherent, dimension shifting gives $\Ext_R^n(K,M)\cong\Ext_R^1(K,E_{n-1})$ for every finitely $n$-presented $R$-module $K$ with $\pd_R(K)\le d$. As $E_{n-1}$ is $\mathsf{FP}_n^{\le d}$-injective, $\Ext_R^1(K,E_{n-1})=0$; hence $M$ is $(n,d)$-injective.
\item (a)$\Rightarrow$(b): The implication follows by taking an injective coresolution of $M$ and applying dimension shifting; note that the assumption that $R$ is left $(n,d)$-coherent is not required in this case.
\item (b)$\Leftrightarrow$(c): Follows from dimension shifting, using the left $(n,d)$-coherence of $R$, and arguing as in the proof of $(a)\Leftrightarrow(b)$.
\end{itemize}
\end{proof}

The proof of the following result is analogous to that of \cite[Prop. 5.2]{PA25} ans is omitted. 

\begin{cor}
Let $R$ be a left $(n,d)$-coherent ring with $n\in\mathbb{N}^*$ and $d\in\mathbb{N}^*\cup\{\infty\}$. Then:
\begin{enumerate}\renewcommand{\labelenumi}{(\alph{enumi})}
\item $\mathsf{FP}_n^{\le d}\text{-}\mathsf{Inj}(R)\cap \left( \mathsf{FP}_n^{\le d}\text{-}\mathsf{Inj}(R)\right )^{\perp_{1}}=\mathsf{Inj}(R)$.
\item $\mathsf{FP}_n^{\le d}\text{-}\mathsf{Proj}(R)\cap {}^{{\perp}_1}\left(\mathsf{FP}_n^{\le d}\text{-}\mathsf{Proj}(R)\right )=\mathsf{Proj}(R)$.
\end{enumerate}    
\end{cor}

We note that part (a) of the preceding corollary is a particular case of a
more general result: whenever a class $\mathcal{C}$ is closed under cokernels
of monomorphisms, the intersection $\mathcal{C}\cap\mathcal{C}^{\perp_1}$
coincides with the class of injective modules; see \cite[Lem. 6.3]{G26}.\\

In the commutative $n$-coherent regular case, the following result is known; see \cite[Prop. 5.14]{PA25}.

\begin{prop}\label{prop: pair projectivo}
Let $R$ be a commutative $(n,d)$-coherent ring with $n\in\mathbb{N}^*$ and $d\in\mathbb{N}^*$. Then 
$\mathsf{FP}_n^{\le d}\text{-}\mathsf{Inj}(R)\bigcap\big(\mathsf{FP}_n^{\le d}(R)\setminus\mathsf{proj}(R)\big) = \emptyset.$
\end{prop}
\begin{proof}
Suppose there exists a non-projective $R$-module $M\in\mathsf{FP}_n^{\le d}(R)$ such that $M \in\mathsf{FP}_n^{\le d}\text{-}\mathsf{Inj}(R)$. Then $\Ext_R^1(M,M)=0$, and consequently $\Ext_R^k(M,M)=0$ for all $k\ge 2$. Since $M$ is finitely $n$-presented and non-projective, we have $1\le\pd_R(M)=s\le d$. However, by \cite[Lem. 1.2]{PT}, this implies $\Ext_R^s(M,M)\neq 0$, a contradiction.
\end{proof}

The following proposition complements Theorem \ref{theo: mah}.

\begin{prop}\label{prop: von neumman}
Let $R$ be a left $(n,d)$-coherent ring with $n\in\mathbb{N}^*$ and $d\in\mathbb{N}^*\cup\{\infty\}$. The following statements are equivalent:
\begin{enumerate}\renewcommand{\labelenumi}{(\alph{enumi})}
\item $\mathsf{FP}_n^{\le d}\text{-}\mathsf{Proj}(R) \subseteq \mathsf{FP}_n^{\le d}\text{-}\mathsf{Inj}(R)$.
\item $\mathsf{FP}_n^{\le d}\text{-}\mathsf{Inj}(R)=\lMod R$.
\item $\Ext^1_R(X,Y)=0$ for all $X,Y\in \mathsf{FP}_n^{\le d}\text{-}\mathsf{Proj}(R)$.
\item $\mathsf{Proj}(R)=\mathsf{FP}_n^{\le d}\text{-}\mathsf{Proj}(R)$.
\item $\mathsf{proj}(R)=\mathsf{FP}_n^{\le d}(R)$.
\item The left annihilator of every finitely generated proper ideal of $R^{\op}$ is nonzero.
\end{enumerate}
\end{prop}
\begin{proof} ~\
\begin{itemize}
\item (a)$\Leftrightarrow$(b)$\Leftrightarrow$(c)$\Leftrightarrow$(d): These equivalences follow from combining \cite[Cor. 4.6]{HMP21} with Corollary \ref{cor:n-d-coh-equiv 1}. 
\item (d)$\Rightarrow$(e) $\Rightarrow$(b): Clear.
\item (e)$\Leftrightarrow$(f): This equivalence follows from Theorem \ref{theo: mah}.
\end{itemize}
\end{proof}

\begin{rmk}
Observe that if $n\in\mathbb{N}^*$ and $d\ge\gD(R)$, 
Proposition \ref{prop: von neumman} provides a characterization of 
classical left \newterm{$n$-von Neumann regular rings}, namely, those rings for which 
every finitely $n$-presented $R$-module is flat (equivalently, projective); 
see \cite[Thm. 3.9]{Zhu} and \cite[Prop. 3.2 and Cor. 5.4]{gp2}.
\end{rmk}

A ring $R$ is called left \newterm{self $\mathsf{FP}_n^{\le d}$-injective} if $R$ is $\mathsf{FP}_n^{\le d}$-injective as a module over itself.

\begin{prop}\label{prop: coh + self}
Let $R$ be a left $(n,d)$-coherent and left self $\mathsf{FP}_n^{\le d}$-injective ring with $n\in\mathbb{N}^*$ and $d\in\mathbb{N}^*\cup\{\infty\}$. Then the following hold:
\begin{enumerate}\renewcommand{\labelenumi}{(\alph{enumi})}
    \item If $d=\infty$, then 
    $\mathsf{FP}_n^{\le d}(R)=\mathsf{proj}(R)\cup\{M \in \mathsf{FP}_n(R) \mid \pd_R(M)=\infty\}.$
    \item If $d<\infty$, then $\mathsf{FP}_n^{\le d}(R)= \mathsf{proj}(R).$
\end{enumerate}
\end{prop}
\begin{proof}
It suffices to show that every finitely $n$-presented $R$-module of finite projective dimension at most $d\in \mathbb{N}^*$ is projective. Let $M\in\mathsf{FP}_n^{\le d}(R)$ be a non-projective module, and assume that $\pd_R(M)=k<\infty$. Since $M$ is not projective, we have $k\ge 1$. Hence, there exists an $R$-module $N$ such that $\Ext_R^k(M,N)\ne 0$. Consider a short exact sequence $L\rightarrowtail P\twoheadrightarrow N $, where $P$ is projective. Consider $\Ext_R^{k}(M,P) \to \Ext_R^{k}(M,N) \to \Ext_R^{k+1}(M,L)=0$. Since $R$ is self $\mathsf{FP}_n^{\le d}$-injective, every projective module is $\mathsf{FP}_n^{\le d}$-injective. By Corollary \ref{cor:n-d-coh-equiv 1}, it follows that $\Ext_R^{k}(M,P)=0$. Consequently, $\Ext_R^{k}(M,N)=0$, a contradiction. Therefore $M$ must be projective. 
\end{proof}

\begin{cor}
Let $R$ be a left $(n,d)$-coherent ring with $n\in\mathbb{N}^*$ and
$d\in\mathbb{N}^*$. Then $R$ is left self $\mathsf{FP}_n^{\le d}$-injective
if and only if the left annihilator of every finitely generated proper ideal of $R^{\op}$ is nonzero.
\end{cor}
\begin{proof}
If $R$ is left self $\mathsf{FP}_n^{\le d}$-injective, the result follows
from Propositions \ref{prop: von neumman} and  \ref{prop: coh + self}. Conversely, assume that left annihilator of every finitely generated proper ideal of $R^{\op}$ is nonzero. By Proposition \ref{prop: von neumman},
$\mathsf{FP}_n^{\le d}\text{-}\mathsf{Inj}(R)=\lMod R$, and hence $R$ is left self $\mathsf{FP}_n^{\le d}$-injective.
\end{proof}

Recall that a cotorsion pair $(\mathcal{A},\mathcal{B})$ in $\lMod R$ is called \newterm{projective} if it is hereditary, complete, and $\mathcal{A}\cap\mathcal{B}=\mathsf{Proj}(R)$. Equivalently,  $(\mathcal{A},\mathcal{B})$ is projective if and only if it is complete, $\mathcal{B}$ contains all projective modules, and $\mathcal{B}$ satisfies the 2-out-of-3 property in $\lMod R$. In this setting, we obtain the following characterization. Note that when $d\ge\gD(R)$, this recovers \cite[Prop. 7.10]{PA25}.

\begin{prop} Let $R$ be a ring, $n\in\mathbb{N}^*$ and $d\in\mathbb{N}^*\cup\{\infty\}$. The following statements are equivalent:
\begin{enumerate}\renewcommand{\labelenumi}{(\alph{enumi})}
\item The cotorsion pair $\left(\mathsf{FP}_n^{\le d}\text{-}\mathsf{Proj}(R), \mathsf{FP}_n^{\le d}\text{-}\mathsf{Inj}(R) \right)$ is projective.
\item Every projective $R$-module is $\mathsf{FP}_n^{\le d}$-injective, and the class $\mathsf{FP}_n^{\le d}\text{-}\mathsf{Inj}(R)$ satisfies the 2-out-of-3 property in $\lMod R$.
\item The ring $R$ is left self $\mathsf{FP}_n^{\le d}$-injective and left $(n,d)$-coherent, and the class $\mathsf{FP}_n^{\le d}\text{-}\mathsf{Inj}(R)$ is closed under kernels of epimorphisms. 
\end{enumerate}
\end{prop}
\begin{proof} ~\
\begin{itemize}
\item (a)$\Leftrightarrow$(b)$\Rightarrow$(c): These implications follow immediately from the definition of projective cotorsion pairs, Corollary \ref{cor:n-d-coh-equiv 1}, and Remark \ref{rmk: complete}.
\item (c)$\Rightarrow$(a): Since $R$ is left self-$\mathsf{FP}_n^{\le d}$-injective, Remark \ref{rmk: complete} implies that every projective $R$-module is $\mathsf{FP}_n^{\le d}$-injective. On the other hand, left $(n,d)$-coherence of $R$ yields that $\mathsf{FP}_n^{\le d}\text{-}\mathsf{Inj}(R)$ is closed under cokernels of epimorphisms (again by Corollary \ref{cor:n-d-coh-equiv 1}). As this class is always closed under extensions, it satisfies the 2-out-of-3 property in $\lMod R$.
\end{itemize}
\end{proof}


\section{$\mathsf{FP}_n^{\le d}$-flat and $\mathsf{FP}_n^{\le d}$-cotorsion modules}\label{s:planos y cotorsion relativos} 
We begin with the following.

\begin{dfn}
Let $R$ be a ring, $n\in\mathbb{N}^*$ and $d\in\mathbb{N}^*\cup\{\infty\}$. An $R^{\op}$-module $F$ is said to be \newterm{$\mathsf{FP}_n^{\le d}$-flat} if $\Tor_1^R(F,K)=0$ for every $K\in\mathsf{FP}_n^{\le d}(R)$.
\end{dfn}

\begin{rmk} Let $R$ be a ring. 
\begin{enumerate}\renewcommand{\labelenumi}{(\alph{enumi})}
\item An $R^{\op}$-module $M$ is $\mathsf{FP}_n^{\le d}$-flat if and only if for every finitely generated
projective $R$-module $P$ and every finitely $(n-1)$-presented submodule
$K\subseteq P$ with $\pd_R(K)\le d-1$, the canonical map $M\otimes_R K\to M\otimes_R P$ is injective. When $d\ge\gD(R)$, this recovers \cite[Thm. 2.13]{Zhu}.
\item Denote by $\mathsf{FP}_n^{\le d}\text{-}\mathsf{Flat}(R^{\op})$ the class of $\mathsf{FP}_n^{\le d}$-flat $R^{\op}$-modules. When $n=1$, this class coincides with the class of $d$-flat modules introduced by Lee \cite{Lee02}. For each $n\in\mathbb{N}^*$, we have the inclusion $\mathsf{FP}_n\text{-}\mathsf{Flat}(R^{\op}) 
\subseteq \mathsf{FP}_n^{\le d}\text{-}\mathsf{Flat}(R^{\op})$, where $\mathsf{FP}_n\text{-}\mathsf{Flat}(R^{\op})$ denotes the class of all $R^{\op}$-modules $F$ such that $\Tor_R^1(F,K)=0$ for every  $K\in\mathsf{FP}_n(R)$. This class was introduced by Zhou in \cite{Zhou04}. By standard properties of $\Tor$, the class $\mathsf{FP}_n^{\le d}\text{-}\mathsf{Flat}(R^{\op})$ is closed under extensions, direct summands and direct limits. If $n\ge 2$, it is also closed under direct products.
\item By \cite[Lem. 2.1]{PPT}, the class
$\varinjlim \mathsf{FP}_n^{\leq d}(R)$ is closed under arbitrary direct
sums, direct limits, pure submodules, pure epimorphic images, and pure
extensions. Moreover, if $n\ge 2$, $n=\infty$, or $d\leq n$, then
$\varinjlim \mathsf{FP}_n^{\leq d}(R)=\bigl[\mathsf{FP}_n^{\leq d}\text{-}\mathsf{Flat}(R^{\op})\bigr]^{\top_1}$.
\end{enumerate}
\end{rmk}

By the \newterm{character module} of an $R$-module (respectively, an $R^{\op}$-module) $M$, we mean the $R^{\op}$-module (respectively, the $R$-module) $M^{\flat}=\Hom_{\mathbb{Z}}(M,\mathbb{Q}/\mathbb{Z})$. Recall that an $R$-module $M$ is flat if and only if its character module $M^{\flat}$ is injective. In our context, we obtain the following characterization.

\begin{prop}\label{prop: carct 2}
Let $R$ be a ring, $n\in\mathbb{N}^*$ and $d\in\mathbb{N}^*\cup\{\infty\}$, and let $M$ be an 
$R^{\op}$-module. Then $M$ is $\mathsf{FP}_n^{\le d}$-flat if and only if its character module $M^{\flat}$ is 
$\mathsf{FP}_n^{\le d}$-injective.
\end{prop}
\begin{proof}
This follows from the standard isomorphism $\Ext_R^1(K,M^{\flat})\cong\Tor_1^R(M,K)^{\flat}$
for every $K\in \mathsf{FP}_n^{\le d}(R)$.
\end{proof}

When $d\ge\gD(R)$, this recovers part (1) of \cite[Thm. 2.15]{Zhu}. For $n\ge 2$, the following 
result allows us to recover part (2) of that theorem.

\begin{prop}\label{prop: carct 1}
Let $R$ be a ring, $n\ge 2$ and $d\in\mathbb{N}^*\cup\{\infty\}$, and let $M$ be an $R$-module. 
Then $M$ is $\mathsf{FP}_n^{\le d}$-injective if and only if  its character module $M^{\flat}$ is $\mathsf{FP}_n^{\le d}$-flat.
\end{prop}
\begin{proof}
Let $K\in\mathsf{FP}_n^{\leq d}(R)$. Since $n\ge 2$, the module $K$ is in particular 
finitely $2$-presented, and hence \cite[Lem. 2.7]{CD96} yields a natural 
isomorphism $\Tor_1^R(M^{\flat},K)\cong\Ext_R^1(K,M)^{\flat}$. 
\end{proof}

Recall that an $R$-module $Q$ is called \newterm{pure-injective} if the functor $\Hom_R(-,Q)$ is exact on pure exact sequences; equivalently, every pure monomorphism $Q\rightarrowtail M$ splits.

\begin{cor}
Let $R$ be a ring, $n\ge 2$, and $d\in\mathbb{N}^* \cup\{\infty\}$. An $R^{\op}$-module $Q$ is $\mathsf{FP}_n^{\le d}$-flat and pure-injective if and only if $Q$ is a direct summand of $M^{\flat}$ for some $\mathsf{FP}_n^{\le d}$-injective $R$-module $M$.
\end{cor}
\begin{proof}
Assume that $Q$ is $\mathsf{FP}_n^{\le d}$-flat and pure-injective. By Propositions \ref{prop: carct 1} and \ref{prop: carct 2}, the module $Q^{\flat\flat}$ is $\mathsf{FP}_n^{\le d}$-flat. The canonical monomorphism $Q\rightarrowtail Q^{\flat\flat}$ is pure. Since $Q$ is pure-injective, it splits. Hence, $Q$ is a direct summand of $Q^{\flat\flat}=(Q^{\flat})^{\flat}$. Conversely, suppose that $Q$ is a direct summand of $M^{\flat}$ for some $\mathsf{FP}_n^{\le d}$-injective $R$-module $M$, say $M^{\flat}=Q\oplus L$. By Proposition \ref{prop: carct 1}, $M^{\flat}$ is $\mathsf{FP}_n^{\le d}$-flat. Since this class is closed under direct summands, $Q$ is $\mathsf{FP}_n^{\le d}$-flat. Finally, $Q$ is pure-injective by \cite[Ch. I, Ex. 42(iii), p. 48]{St74}.
\end{proof}

Recall that a pair of classes $\mathcal{M}\subseteq\lMod{R^{\op}}$ and $\mathcal{C}\subseteq \lMod R$ form a \newterm{duality pair} $(\mathcal{M},\mathcal{C})$ if the following conditions hold:
\begin{itemize}
\item $M\in\mathcal{M}$ if and only if $M^{\flat}\in\mathcal{C}$;
\item $\mathcal{C}$ is closed under direct summands and finite direct sums.
\end{itemize}

\begin{prop}\label{prop: pureza}
Let $R$ be a ring, $n\ge 2$, and $d\in\mathbb{N}^*\cup\{\infty\}$. Then:
\begin{enumerate}\renewcommand{\labelenumi}{(\alph{enumi})}
\item The pair 
$\bigl(\mathsf{FP}_n^{\le d}\text{-}\mathsf{Flat}(R^{\op}),\mathsf{FP}_n^{\le d}\text{-}\mathsf{Inj}(R) \bigr)$
is a duality pair. Moreover, the class 
$\mathsf{FP}_n^{\le d}\text{-}\mathsf{Flat}(R^{\op})$ 
is closed under pure submodules, pure quotients, and pure extensions, and it is preenveloping and covering.
\item The pair $\bigl( \mathsf{FP}_n^{\le d}\text{-}\mathsf{Inj}(R), 
\mathsf{FP}_n^{\le d}\text{-}\mathsf{Flat}(R^{\op}) \bigr)$
is a duality pair. Moreover, the class 
$\mathsf{FP}_n^{\le d}\text{-}\mathsf{Inj}(R)$ 
is closed under pure submodules, pure quotients, and pure extensions, and it is preenveloping and covering.
\end{enumerate}
\end{prop}
\begin{proof}
The duality assertions follow from Propositions \ref{prop: carct 1} and \ref{prop: carct 2}. The closure properties follow from \cite[Thm. 3.1]{HJ08}. Since $\mathsf{FP}_n^{\le d}\text{-}\mathsf{Flat}(R^{\op})$ is closed under direct products and  $\mathsf{FP}_n^{\le d}\text{-}\mathsf{Inj}(R)$ is closed under direct sums, it follows again from \cite[Thm. 3.1]{HJ08} that  $\mathsf{FP}_n^{\le d}\text{-}\mathsf{Flat}(R^{\op})$ is preenveloping and covering and that  $\mathsf{FP}_n^{\le d}\text{-}\mathsf{Inj}(R)$ is preenveloping and covering. 
\end{proof}

When $n\ge 2$ and $d\ge\gD(R)$, the previous result recovers \cite[Cor. 3.7 and Prop. 3.10(4)]{BP} and \cite[Props. 3.4,3.5]{BEI18}.

\begin{prop}\label{cor:perfectpair}
Let $R$ be a ring, $n\geq 2$ and $d\in\mathbb{N}^*\cup\{\infty\}$. Then $R$ is left self $\mathsf{FP}_n^{\le d}$-injective if and only if the pair $\bigl( \mathsf{FP}_n^{\le d}\text{-}\mathsf{Inj}(R),\mathsf{FP}_n^{\le d}\text{-}\mathsf{Inj}(R)^{\perp_1} \bigr)$
is a perfect cotorsion pair.
\end{prop}
\begin{proof}
The proof is analogous to that of \cite[Thm. 4.4]{BP}, with Proposition \ref{prop: pureza} and Remark \ref{rmk: complete} replacing the corresponding results.
\end{proof}

We now study the $\mathsf{FP}_{n}^{\le d}$-flatness and $\mathsf{FP}_{n}^{\le d}$-injectivity of $\Hom$-modules.
The techniques used in this section are inspired by the work of Bo and Zhongkui \cite{BZ12}.

\begin{theorem}\label{thm:4.1}
Let $R$ be a commutative ring, $n\ge 2$ and $d\in\mathbb{N}^*\cup\{\infty\}$. Then the following statements hold:
\begin{enumerate}\renewcommand{\labelenumi}{(\alph{enumi})}
\item $\Hom_R(A,B)$ is $\mathsf{FP}_{n}^{\le d}$-flat for every $\mathsf{FP}_{n}^{\le d}$-injective $R$-module $A$ and every injective $R$-modules $B$.
\item $\Hom_R(A,B)$ is $\mathsf{FP}_{n}^{\le d}$-flat for every projective $R$-module $A$ and every $\mathsf{FP}_{n}^{\le d}$-flat $R$-module $B$.
\end{enumerate}
\end{theorem}
\begin{proof} ~\
\begin{itemize}
\item For part \textbf{(a)}, let $K\in\mathsf{FP}_n^{\le d}(R)$, and fix an exact sequence
$Q\rightarrowtail F\twoheadrightarrow K$ with $F$ finitely generated and free.  Since $A$ is an $\mathsf{FP}_n^{\le d}$-injective $R$-module, applying $\Hom_R(-, A)$ yields the exact sequence
$\Hom_R(K,A)\rightarrowtail \Hom_R(F,A)\twoheadrightarrow \Hom_R(Q,A)$. 
Now let $B$ be an injective $R$-module. Applying $\Hom_R(-,B)$ gives an exact sequence 
$\Hom_R(\Hom_R(Q,A),B)\rightarrowtail \Hom_R(\Hom_R(F,A),B)\twoheadrightarrow \Hom_R(\Hom_R(K,A),B)$. 
By \cite[Thm. 2.6.13]{WK16}, and using that $n\ge 2$, this exactness is
equivalent to the exactness of the sequence
$\Hom_R(A,B)\otimes_R Q\rightarrowtail \Hom_R(A,B)\otimes_R F\twoheadrightarrow \Hom_R(A,B)\otimes_R K$. Therefore, $\Hom_R(A,B)$ is $\mathsf{FP}_{n}^{\le d}$-flat.
\item For part \textbf{(b)}, let $A$ be a projective $R$-module. Then there exists a projective
$R$-module $Q$ such that $A\oplus Q\cong\bigoplus_{i\in I} R$
for some index set $I$. For an $\mathsf{FP}_n^{\le d}$-flat $R$-module $B$, we have
$\Hom_R(A,B)\oplus\Hom_R(Q,B)\cong\prod_{i\in I}B$. Since $\mathsf{FP}_n^{\le d}$ flat $R$-modules are closed under products and direct summands, $\Hom_R(A,B)$ is $\mathsf{FP}_n^{\le d}$-flat.
\end{itemize}
\end{proof}

\begin{cor}\label{cor:4.2}
Let $R$ be a commutative ring, $n\ge 2$, and $d\in\mathbb{N}^*\cup\{\infty\}$. Then:
\begin{enumerate}\renewcommand{\labelenumi}{(\alph{enumi})}
\item $\Hom_R(A,B)$ is $\mathsf{FP}_{n}^{\le d}$-flat for all injective $R$-modules $A$ and $B$.
\item $\Hom_R(A,B)$ is $\mathsf{FP}_{n}^{\le d}$-flat for all projective $R$-modules $A$ and $B$.
\end{enumerate}
\end{cor}

Recall that, for any ring $R$, the class of flat $R$-modules coincides 
with the class of $\mathsf{FP}_1$-flat modules. Moreover, if $R$ is 
commutative, coherence of $R$ is characterized by the property that 
$\Hom_R(A,B)$ is flat for all injective $R$-modules $A$ and $B$.  In particular, when $d\ge\gD(R)$, Corollary \ref{cor:4.2} yields $\mathsf{FP}_2$-flat modules that need not be $\mathsf{FP}_1$-flat over non-coherent commutative rings. 

\begin{cor}\label{cor:4.3}
Let $R$ be a commutative ring, $n\ge 2$, and $d\in\mathbb{N}^{*}\cup\{\infty\}$.
For an $R$-module $M$, the following statements are equivalent:
\begin{enumerate}\renewcommand{\labelenumi}{(\alph{enumi})}
\item $M$ is $\mathsf{FP}_n^{\le d}(R)$-injective.
\item $\Hom_R(M,E)$ is $\mathsf{FP}_n^{\le d}(R)$-flat for every injective $R$-module $E$.
\item $M\otimes_R E$ is $\mathsf{FP}_n^{\le d}(R)$-injective for every flat $R$-module $E$.
\end{enumerate}
\end{cor}
\begin{proof} ~\
\begin{itemize}
\item (a)$\Rightarrow$(b): This implication follows from Theorem \ref{thm:4.1}.
\item (b)$\Rightarrow$(c): For any flat $R$-module $E$, we have $(M \otimes_R E)^{\flat}\cong\Hom_R(M, E^{\flat}),$ and since $E^{\flat}$ is injective, $\Hom_R(M, E^{\flat})$ is $\mathsf{FP}_n^{\le d}(R)$-flat by (b). Hence, by Proposition \ref{prop: carct 1}, $M \otimes_R E$ is $\mathsf{FP}_n^{\le d}(R)$-injective.
\item (c)$\Rightarrow$(a): Apply (c) with $E=R$.
\end{itemize}
\end{proof}

\begin{theorem}\label{thm:4.5}
Let $R$ be a commutative ring, $n \ge 2$, and $d\in\mathbb{N}^*\cup \{\infty\}$. The following statements are equivalent:
\begin{enumerate}\renewcommand{\labelenumi}{(\alph{enumi})}
\item $R$ is self $\mathsf{FP}_n^{\le d}(R)$-injective.
\item Every injective $R$-module is $\mathsf{FP}_n^{\le d}(R)$-flat.
\item Every flat $R$-module is $\mathsf{FP}_n^{\le d}(R)$-injective.
\item $\Hom_R(A,B)$ is $\mathsf{FP}_n^{\le d}(R)$-injective for all free $R$-modules $A$ and $B$.
\item $\Hom_R(A,B)$ is $\mathsf{FP}_n^{\le d}(R)$-injective for all projective $R$-modules $A$ and $B$.
\item $\Hom_R(A,B)$ is $\mathsf{FP}_n^{\le d}(R)$-injective for all projective $R$-modules $A$ and flat $R$-modules $B$.
\item $\Hom_R(A,B)$ is $\mathsf{FP}_n^{\le d}(R)$-flat for all flat $R$-modules $A$ and injective $R$-modules $B$.
\end{enumerate}
\end{theorem}
\begin{proof} ~\
\begin{itemize}
\item (a)$\Rightarrow$(b): This implications follows by Corollary \ref{cor:4.3}.
\item (b)$\Rightarrow$(c): Let $M$ be a flat $R$-module. Then $M^{\flat}$ is injective and hence $\mathsf{FP}_n^{\le d}(R)$-flat by (b). Proposition \ref{prop: carct 1} then implies that $M$ is $\mathsf{FP}_n^{\le d}(R)$-injective. 
\item (c)$\Rightarrow$ (f): Let $A$ be projective and $B$ flat. There exists a projective module $C$ such that $A\oplus C\cong\bigoplus_{i\in I} R$. Then $\Hom_R(A,B)\oplus\Hom_R(C,B)\cong\prod_{i\in I}B$. Since $B$ is $\mathsf{FP}_n^{\le d}(R)$-injective by (c), so is $\Hom_R(A,B)$. 
\item (f)$\Rightarrow$(e)$\Rightarrow$(d): Immediate.
\item (d)$\Rightarrow$(a): This follows by taking $A=B=R$.
\item (g)$\Rightarrow$(b): This follows by taking $A=R$. 
\item (b)$\Rightarrow$(g): If $A$ is flat and $B$ injective, then $\Hom_R(A,B)$ is injective by \cite[Thm. 2.5.5]{WK16},and hence $\mathsf{FP}_n^{\le d}(R)$-flat by (b).
\end{itemize}
\end{proof}

\begin{prop}
Let $R$ be a commutative ring, $n\in\mathbb{N}^*$, and $d\in\mathbb{N}^*\cup\{\infty\}$. For an $R$-module $M$, the following statements are equivalent:
\begin{enumerate}\renewcommand{\labelenumi}{(\alph{enumi})}
\item $M$ is $\mathsf{FP}_n^{\le d}(R)$-flat.
\item $\Hom_R(M,E)$ is $\mathsf{FP}_n^{\le d}(R)$-injective for every injective $R$-module $E$.
\item $M\otimes_R E$ is $\mathsf{FP}_n^{\le d}(R)$-flat for every flat $R$-module $E$.
\end{enumerate}
\end{prop}
\begin{proof} ~\
\begin{itemize}
\item (a)$\Rightarrow$(b): Let $E$ be an injective $R$-module and let 
$K\in\mathsf{FP}_n^{\le d}(R)$. By the standard adjunction, 
$\Ext_R^1(K,\Hom_R(M,E))\cong\Hom_R(\Tor_1^R(M,K),E)$.
By (a), $\Tor_1^R(M,K)=0$, and hence $\Hom_R(M,E)$ is
$\mathsf{FP}_n^{\le d}$-injective.
\item  (b)$\Rightarrow$(c): Let $E$ be a flat $R$-module. By Proposition \ref{prop: carct 2}, 
it suffices to show that $(M\otimes_R E)^{\flat}$ is $\mathsf{FP}_n^{\le d}(R)$-injective. Since $(M\otimes_R E)^{\flat}\cong\Hom_R(M,E^{\flat})$ and $E^{\flat}$ is injective, this follows from (b). 
\item (c)$\Rightarrow$(a): This follows by taking $E=R$.
\end{itemize}
\end{proof}

\begin{dfn} Let $R$ be a ring, and $n\in\mathbb{N}^{*}$ and $d\in\mathbb{N}^{*}\cup\{\infty\}$. 
An $R^{\op}$-module $M$ is said to be \newterm{$\mathsf{FP}_n^{\le d}$-cotorsion} if $\Ext_R^1(F, M)=0$ for every $\mathsf{FP}_n^{\le d}$-flat $R^{\op}$-module $F$. 
\end{dfn}

We denote this class by $\mathsf{FP}_n^{\le d}\text{-}\mathsf{Cot}(R^{\op})$. When $d\geq\gD(R)$, these modules coincide with the \newterm{$(n,0)$-cotorsion} modules studied by Zhu in \cite{Zhu17}. More generally, taking $\mathcal{C}=\mathsf{FP}_n^{\le d}(R)$, the $\mathsf{FP}_n^{\le d}$-cotorsion modules are precisely the \newterm{$\mathcal{C}$-cotorsion} modules considered by Zhu in \cite{Zhu21}. In particular, \cite[Thm. 1]{Zhu21} gives the following characterization.

\begin{prop}
Let $R$ be a ring, $n\in\mathbb{N}^*$, $d\in\mathbb{N}^*\cup \{\infty\}$, and let $M$ be an $R^{\op}$-module. Then the following statements are equivalent:
\begin{enumerate}\renewcommand{\labelenumi}{(\alph{enumi})}
\item $M$ is $\mathsf{FP}_n^{\le d}$-cotorsion.
\item $M$ is injective with respect to every short exact sequence $K\rightarrowtail P\twoheadrightarrow F$ of $R^{\op}$-modules with $F\in\mathsf{FP}_n^{\le d}\text{-}\mathsf{Flat}(R^{\op})$.
\item Every $F\in\mathsf{FP}_n^{\le d}\text{-}\mathsf{Flat}(R^{\op})$ is projective with respect to every short exact sequence $M\rightarrowtail M'\twoheadrightarrow M''$ of $R^\op$-modules.
\end{enumerate}
Moreover, if the injective envelope $E(M)$ belongs to $\mathsf{FP}_n^{\le d}\text{-}\mathsf{Flat}(R^{\op})$, then these conditions are also equivalent to:
\begin{enumerate}
\item[(e)] For every exact sequence $M\rightarrowtail F \twoheadrightarrow L$ with $F\in\mathsf{FP}_n^{\le d}\text{-}\mathsf{Flat}(R^{\op})$, the morphism $F \to L$ is an $\mathsf{FP}_n^{\le d}$-flat precover of $L$.
\item[(f)] $M$ is the kernel of some $\mathsf{FP}_n^{\le d}$-flat precover $E\to L$, where $E$ is injective.
\end{enumerate}
\end{prop}

By Proposition \ref{prop: pureza}, $\mathsf{FP}_n^{\le d}\text{-}\mathsf{Inj}(R)$ is preenveloping for $n\ge 2$. If $\mathsf{FP}_n^{\le d}\text{-}\mathsf{Proj}(R)$ is closed under pure quotients, we show below that it is enveloping. For $d\ge\gD(R)$, this is know by \cite[Prop. 3.7]{BEI18}.

\begin{prop}
Let $R$ be a ring, $n\ge 2$ and $d\in\mathbb{N}^{*}\cup\{\infty\}$. The following statements are equivalent:
\begin{enumerate}\renewcommand{\labelenumi}{(\alph{enumi})}
\item $\mathsf{FP}_n^{\le d}\text{-}\mathsf{Proj}(R)$ is closed under pure quotients.
\item $\big(\mathsf{FP}_n^{\le d}\text{-}\mathsf{Proj}(R), \mathsf{FP}_n^{\le d}\text{-}\mathsf{Cot}(R^{\op})\big)$ is a duality pair.
\end{enumerate}
If these conditions hold, then $\mathsf{FP}_n^{\le d}\text{-}\mathsf{Inj}(R)$ is enveloping.
\end{prop}
\begin{proof}
The proof is similar to that of \cite[Prop. 3.7]{BEI18}, replacing the corresponding statements with 
Remark \ref{rmk: complete} and Propositions \ref{prop: carct 2} and \ref{prop: carct 1}.
\end{proof}

By \cite[Prop. 3.2(2)]{BEI18}, for any class $\mathcal{B}$ of $R^{\op}$-modules, $\Ker\Tor_1^R(\mathcal{B}, -)$ is the first half of a perfect cotorsion pair in $\lMod R$. Taking $\mathcal{B}=\mathsf{FP}_n^{\le d}(R^{\op})$ gives the following consequence. For $d\ge\gD(R)$, this recovers \cite[Prop. 3.3]{BEI18}.

\begin{prop}
Let $R$ be a ring, $n\in \mathbb{N}^*$, and $d\in\mathbb{N}^*\cup\{\infty\}$. Then 
$\bigl(\mathsf{FP}_n^{\le d}\text{-}\mathsf{Flat}(R^{}),\mathsf{FP}_n^{\le d}\text{-}\mathsf{Cot}(R^{})\bigr)$
is a perfect cotorsion pair.
\end{prop}

We next study $\mathsf{FP}_n^{\le d}$-flat and $\mathsf{FP}_n^{\le d}$-cotorsion modules over $(n,d)$-coherent rings.

\begin{theorem}\label{cor:n-d-coh-equiv 2}
Let $R$ be a ring, $n\ge 2$ and $d\in\mathbb{N}^{*}\cup\{\infty\}$. The following statements are equivalent:
\begin{enumerate}
\item $R$ is left $(n,d)$-coherent.
\item[(26)] $\Tor^{R}_{j+1}(M,C)=0$ for every integer $j\ge 0$, every $C\in\mathsf{FP}_{n}^{\le d}(R)$, and every  
 $\mathsf{FP}_{n}^{\le d}(R)$-flat  $R^{\op}$-module $M$.
 \item[(27)] $\Tor^{R}_{2}(M,C)=0$ for every $C\in\mathsf{FP}_{n}^{\le d}(R)$ and every  
$\mathsf{FP}_{n}^{\le d}(R)$-flat  $R^{\op}$-module $M$.
\item[(28)] Every $\mathsf{FP}_{n+1}^{\le d}$-flat $R^{\op}$-module is $\mathsf{FP}_{n}^{\le d}$-flat.
\item[(29)] Every \(\mathsf{FP}_n^{\le d}\)-cotorsion $R^{\op}$-module is \(\mathsf{FP}_{n+1}^{\le d}\)-cotorsion.
\item[(30)] The pair $(\mathsf{FP}_n^{\le d}\text{-}\mathsf{Flat}(R^{\op}), \mathsf{FP}_n^{\le d}\text{-}\mathsf{Cot}(R^{\op}))$ is a hereditary cotorsion pair. 
\end{enumerate}
\end{theorem}
\begin{proof} ~\
\begin{itemize}
\item (1)$\Leftrightarrow$(26)$\Leftrightarrow$(27): These equivalences follow from Theorem \ref{thm:n-coh 1} and \cite[Thm. 1]{Zhu2017}.
\item (1)$\Rightarrow$(28) $\Rightarrow$(29): Immediate.
\item (29)$\Rightarrow$(28): Since
$\mathsf{FP}_n^{\le d}\text{-}\mathsf{Flat}(R^{\op})\subseteq \mathsf{FP}_{n+1}^{\le d}\text{-}\mathsf{Flat}(R^{\op})$,
we have $\mathsf{FP}_{n+1}^{\le d}\text{-}\mathsf{Cot}(R^{\op})\subseteq\mathsf{FP}_n^{\le d}\text{-}\mathsf{Cot}(R^{\op})$. By hypothesis,
the reverse inclusion also holds and hence
$\mathsf{FP}_{n+1}^{\le d}\text{-}\mathsf{Cot}(R^{\op})
=\mathsf{FP}_n^{\le d}\text{-}\mathsf{Cot}(R^{\op})$. Since $\big(\mathsf{FP}_n^{\le d}\text{-}\mathsf{Flat}(R^{\op}),
\mathsf{FP}_n^{\le d}\text{-}\mathsf{Cot}(R^{\op})\big)$ is a cotorsion pair, the result follows.
\item (28)$\Rightarrow$(1): By Corollary \ref{cor:n-d-coh-equiv 1}, it suffices to show that every
$\mathsf{FP}_{n+1}^{\le d}$-injective $R$-module is $\mathsf{FP}_{n}^{\le d}$-injective.
Let $M$ be an $\mathsf{FP}_{n+1}^{\le d}$-injective. By Proposition \ref{prop: carct 1}, $M^{\flat}$
is $\mathsf{FP}_{n+1}^{\le d}$-flat. By assumption, $M^{\flat}$ is $\mathsf{FP}_{n}^{\le d}$-flat and Proposition \ref{prop: carct 1} then implies that $M$ is $\mathsf{FP}_{n}^{\le d}$-injective.
\item (27)$\Rightarrow$(30): It suffices to show that 
$\mathsf{FP}_n^{\le d}\text{-}\mathsf{Flat}(R^{\op})$ 
is resolving. It is closed under extensions and contains all projective modules. Let $A\rightarrowtail B\twoheadrightarrow C$ be exact, with $B,C \in \mathsf{FP}_n^{\le d}\text{-}\mathsf{Flat}(R^{\op})$.
For $M\in\mathsf{FP}_n^{\le d}(R)$, the long exact $\Tor$ sequence gives
$\Tor_2^R(C,M)\to\Tor_1^R(A,M)
\to \Tor_1^R(B,M)$. The outer terms vanish, so $\Tor_1^R(A,M)=0$. Hence $A$ is $\mathsf{FP}_n^{\le d}$-flat.
\item (30)$\Rightarrow$(1): By Corollary \ref{cor:n-d-coh-equiv 1}, it suffices to prove that 
$\mathsf{FP}_n^{\le d}\text{-}\mathsf{Inj}(R)$ 
is closed under cokernels of monomorphisms. Let
$A\rightarrowtail B\twoheadrightarrow C$ be exact, with  $A,B\in\mathsf{FP}_n^{\le d}\text{-}\mathsf{Inj}(R)$. Then 
$C^{\flat} \rightarrowtail B^{\flat} \twoheadrightarrow A^{\flat}$ is exact, adn Proposition \ref{prop: carct 1} gives 
$A^{\flat}, B^{\flat}\in \mathsf{FP}_n^{\le d}\text{-}\mathsf{Flat}(R^{\op})$. 
Since $\big(\mathsf{FP}_n^{\le d}\text{-}\mathsf{Flat}(R^{\op}),\,
\mathsf{FP}_n^{\le d}\text{-}\mathsf{Cot}(R^{\op})
\big)$ is hereditary, $C^{\flat}\in\mathsf{FP}_n^{\le d}\text{-}\mathsf{Flat}(R^{\op})$. Proposition \ref{prop: carct 1} then gives  
$C\in\mathsf{FP}_n^{\le d}\text{-}\mathsf{Inj}(R)$.
\end{itemize}
\end{proof}

In particular, if $n\ge 2$ and $d\ge\gD(R)$, this recovers the equivalence of $(1)$, $(2)$, and $(4)$ in \cite[Thm. 2.10]{Zhu17}, as well as that of $(1)$ and $(12)$ in \cite[Thm. 2.8]{Zhu17}.\\

Recall that the \newterm{finitistic dimension} of a ring $R$, denoted $\mathsf{fin.dim}(R)$, is the supremum of the projective dimensions of all $R$-modules that admit a finite resolution by finitely generated projective modules. Given a ring $R$ and an $R$-$R$-bimodule $M$, the \newterm{trivial extension} $R \ltimes M$ is the ring with underlying abelian group $R\oplus M$ and multiplication $(r,m)(r',m') = (rr',rm' + mr')$. The following corollary is an immediate consequence of Theorem \ref{cor:n-d-coh-equiv 2} and \cite[Cor. 3.19]{CX17}.

\begin{cor}
Let $R$ be a left $(n,d)$-coherent ring with $n\ge 2$ and $d\in\mathbb{N}^*$.
Let $\lambda\colon R\twoheadrightarrow S$ be a ring epimorphism, and let $M$ be an $S$-bimodule such that
$M\in\mathsf{FP}_n^{\le d}\text{-}\mathsf{Flat}(R^{\op})$ and  $S\in\mathsf{FP}_n^{\le d}(R).$
Then the following inequalities hold:
\begin{enumerate}\renewcommand{\labelenumi}{(\alph{enumi})}
\item $\mathsf{fin.dim}(S \ltimes M)\le \mathsf{fin.dim}(S) + \mathsf{fin.dim}(R \ltimes M) + 1$.
\item $\mathsf{fin.dim}(S)\le \mathsf{fin.dim}(R) + \mathsf{fin.dim}(S \ltimes M)$.
\end{enumerate}
\end{cor}

\section{Acknowledgements}
The author would like to thank Marco Pérez and Viviana Gubitosi for their helpful comments and valuable discussions.

\end{document}